\documentclass[11pt]{article}
\usepackage{amssymb}
\usepackage{babel}[english]
\usepackage{amsmath,amsfonts,amsthm,amssymb}
\usepackage[dvips]{graphics}
\usepackage{epsfig}
\usepackage{indentfirst}
\usepackage{color}
\usepackage{amsmath}
\usepackage{relsize}
\usepackage{comment}
\allowdisplaybreaks

\newtheoremstyle{theorem}
  {10pt}          
  {10pt}  
  {\sl}  
 {}
  {\bf}  
  {. }    
  { }    
  {}     
\theoremstyle{theorem}

\newtheorem{theorem}{Theorem}[section]
\newtheorem{corollary}{Corollary}[section]
\newtheorem{definition}{Definition}[section]

\newtheorem{remark}{Remark}[section]

\numberwithin{equation}{section}

\newtheoremstyle{defi}
{10pt}  
{10pt}  
{\rm}   
{}      
{\bf}   
{. }    
{ }     
{}      
\theoremstyle{defi}

\definecolor{mypink}{RGB}{219, 48, 122}
\definecolor{myblue}{RGB}{0, 0, 122}

\begin{document}

\title{Relative entropy inequality for capillary fluids with density dependent viscosity and applications}

\author{Matteo Caggio$^{1}$ \ \ Donatella Donatelli$^{2}$
  \\
  {\small  1. Institute of Mathematics of the Academy of Sciences of the Czech Republic,} \\
  {\small \v Zitn\' a 25, 11567, Praha 1, Czech Republic}\\
  {\small  2. Department of Information Engineering, Computer Science and Mathematics,}\\
  {\small University of L'Aquila}\\
  {\small via Vetoio, Coppito - 67100 L' Aquila, Italy}\\
  {\small caggio@math.cas.cz}\\
  {\small donatella.donatelli@univaq.it}
\date{}
}

\maketitle
\begin{abstract}
    \noindent
	We derive a relative entropy inequality for capillary compressible fluids with density dependent viscosity.
	Applications in the context of weak-strong uniqueness analysis, pressureless fluids and high-Mach number flows are presented.
\end{abstract}
\smallskip


\tableofcontents{}

\newpage{}

\section{Introduction}
\noindent
We consider the following compressible Navier-Stokes-Korteweg system in $(0,T) \times \mathbb{T}^3$:
\begin{equation} \label{cont_cap}
    \partial_t \varrho + \textrm{div} (\varrho  u)=0
\end{equation}
\begin{equation*}
\partial_t (\varrho u) + \textrm{div} (\varrho u \otimes u) + \nabla p(\varrho) - 2\nu \textrm{div} (\varrho Du)
\end{equation*}
\begin{equation} \label{mom_cap}
    - 2\kappa^2 \left[ 
    \nabla (\varrho  \Delta \varrho) + \frac{1}{2} \nabla ( |\nabla \varrho |^2)
    - 4 \textrm{div}(\varrho \nabla \sqrt{\varrho} \otimes \nabla \sqrt{\varrho})
    \right]=0,
\end{equation}
with initial data
\begin{equation} \label{idata}
    \varrho(0,x) = \varrho_0(x), \ \ \ \ \varrho  u (0,x) = \varrho_0(x) u_0(x).
\end{equation}
Here, $\mathbb{T}^3$ denotes the three-dimensional flat torus $[0,1]^3$, the function $\varrho=\varrho (x,t)$ represents the density of the fluid and $u=u(x,t)$ the three-dimensional velocity field, functions of the spatial position $x \in \mathbb{T}^3$ and the time $t$. The pressure $p$ is of the type  $p(\varrho) = \varrho ^\gamma$ with $\gamma > 1$ physical constant. System (\ref{cont_cap}) - (\ref{idata}) describes compressible viscous capillary fluids. 

Capillary fluids belong to the more general class of Korteweg fluids whose form is determined by the following system of equations,
\begin{equation} \label{cont_nsk}
    \partial_t \varrho+ \textrm{div} (\varrho  u) = 0,
\end{equation}
\begin{equation} \label{mom_nsk}
    \partial_t (\varrho u) + \textrm{div} (\varrho u \otimes u) + \nabla p(\varrho ) = 2 \textrm{div} \mathbb{S} + 2\kappa^2 \textrm{div} \mathbb{K}.
\end{equation} 
The quantity 
$\mathbb{S}$ is the viscous stress tensor given by
\begin{equation} \label{S}
    \mathbb{S} = \mu(\varrho)Du+\lambda(\varrho)\textrm{div}u\mathbb{I}, \quad Du=\frac{\nabla u+\nabla^{t}u}{2}
\end{equation}
and
\begin{equation} \label{K}
    \mathbb{K}
    =
    \left(\varrho \textrm{div}(k(\varrho)\nabla\varrho)-\frac{1}{2}(\varrho k'(\varrho)-k(\varrho ))|\nabla\varrho|^2\right)\mathbb{I}-
\left(k(\varrho )\nabla\varrho \otimes\nabla\varrho\right).
\end{equation}
is the Korteweg tensor with $k(\varrho)$ the capillarity coefficient. The viscosity coefficients $\mu(\varrho)$ and $\lambda(\varrho)$ satisfy
$$
\mu(\varrho)\geq0, \ \ \mu(\varrho)+ 3\lambda(\varrho)\geq0.
$$
The system (\ref{cont_cap}) - (\ref{mom_cap}) can be recast from (\ref{cont_nsk}) - (\ref{K}) by choosing
$\lambda(\varrho)=0$ and $k(\varrho)=\kappa^2$, with $\kappa > 0$ constant capillarity coefficient.

The aim of this paper is to prove a relative entropy inequality for the system (\ref{cont_cap}) - (\ref{mom_cap}) and to discuss  some applications. 
Relative entropies are non-negative quantities that provide a sort of distance between two solutions of the same problem/pde, one of which typically enjoys some very smooth regularity properties. The method of relative entropy or modulated energy was introduced in the pioneering paper by Dafermos \cite{Daf} and since then it has  been  applied to partial differential equations of different types and has become a successful tool in the study of weak strong uniqueness properties (see \cite{FeNoJi}, \cite{DoFeMa}) and in scaling limit analysis such as incompressible limits, vanishing viscosity limits, quasi-neutral limits, high friction limits,  see for example \cite{CaDoNeSun}, \cite{DoFe}, \cite{CianLat} and references therein.
In this paper we will focus on the following two applications for capillary fluids: weak-strong uniqueness, high Mach number limit.

Weak-strong uniqueness means that a weak and strong solution emanating from the same initial data coincide as long as the latter exists (see Theorem \ref{main-r1}).
Applications related to this analysis will concern also  pressureless capillary flows. In particular, we will consider the following pressureless system, 
\begin{equation} \label{cont_cap-press}
    \partial_t \varrho + \textrm{div} (\varrho u) = 0,
\end{equation}
\begin{equation*}
\partial_t (\varrho u) + \textrm{div} (\varrho u \otimes u) - 2\nu \textrm{div}(\varrho Du)
\end{equation*}
\begin{equation} \label{mom_cap-press}
    - 2\kappa^2 \left[ 
    \nabla (\varrho \Delta \varrho) + \frac{1}{2} \nabla ( |\nabla \varrho|^2)
    - 4 \textrm{div}(\varrho \nabla \sqrt{\varrho} \otimes \nabla \sqrt{\varrho})
    \right]=0;
\end{equation}
 obtained by the re-scaled version of the  system (\ref{cont_cap}) - (\ref{mom_cap}) in terms of the Mach number $\mathcal{M}a$, namely 
\begin{equation} \label{cont_cap-rs}
    \partial_t \varrho_{\varepsilon} + \textrm{div} (\varrho_{\varepsilon} u_{\varepsilon}) = 0,
\end{equation}
\begin{equation*}
\partial_t (\varrho_{\varepsilon} u_{\varepsilon}) + \textrm{div} (\varrho_{\varepsilon} u_{\varepsilon} \otimes u_{\varepsilon}) + \varepsilon^2 \nabla p(\varrho_{\varepsilon}) - 2 \textrm{div} (\mu(\varrho_{\varepsilon}) Du_{\varepsilon})
\end{equation*}
\begin{equation} \label{mom_cap-rs}
    - 2\kappa^2 \left[ 
    \nabla (\varrho_{\varepsilon} \Delta \varrho_{\varepsilon}) + \frac{1}{2} \nabla ( |\nabla \varrho_{\varepsilon}|^2)
    - 4 \textrm{div}(\varrho_{\varepsilon} \nabla \sqrt{\varrho_{\varepsilon}} \otimes \nabla \sqrt{\varrho_{\varepsilon}})
    \right]=0,
\end{equation}
where
\begin{equation*}
    \varepsilon^2 = \frac{1}{(\mathcal{M}a)^2},
\end{equation*}
and by performing the high Mach number limit $\mathcal{M}a\to \infty$.

For the system (\ref{cont_cap-press}) - (\ref{mom_cap-press}) we will discuss the weak-strong uniqueness property as a consequence of the analysis performed for the system with the presence of the pressure (see Theorem \ref{main-r2}). Moreover, we will also take into account a particular case where irrotational initial data for the velocity field are considered (see Corollary \ref{cor}). This result is somehow relevant since, as we will see in the Section \ref{WSpl}, we obtain a weak-strong uniqueness  results by considering an initial density which is merely $L^{1}$ in contrast with the higher regularity usually required in this framework.

The present analysis could be seen both as a natural continuation of the previous analysis in \cite{CaDo}, in which the existence of weak solutions and related properties for the pressureless system, the high Mach number limit have been proved, and an extension  of the weak-strong uniqueness result developed in \cite{BGL} for quantum fluids. The interest in the weak-strong uniqueness property relates to the fact that is new both in the context of relative energy estimates for the high-Mach number or pressureless regime and in the context of the existence theory for capillary fluids with density dependent viscosity. Indeed,  while in \cite{CaDo} (see also \cite{BrDeLi}) the definition of the weak solution is based on a particular choice of the test function that prevents the weak-strong uniqueness analysis in the usual sense, the recent advancement in the theory concerning the existence of the weak solutions (see Section 2.2 below), allows the study of the weak-strong uniqueness property as commonly known (see, for example, \cite{FeNoJi}).

A further application of the relative entropy inequality for the system (\ref{cont_cap-rs}) - (\ref{mom_cap-rs}) concerns the rigorous proof of the high Mach number limit. Indeed we will study the convergence of a weak solution of (\ref{cont_cap-rs}) - (\ref{mom_cap-rs}) towards the strong solution of the corresponding pressureless system in the limit as $\varepsilon \to 0$ (see Theorem \ref{main-r3}). This result can be considered as the completion of the analysis of the high Mach number limit  for capillary fluids started in \cite{CaDo} where the authors rigorously proved a weak-weak convergence result in the high compressible regime.

This manuscript is organised as follows. In Section \ref{pre} we introduce the notation and we discuss the existence of a weak solution to the system (\ref{cont_cap}) - (\ref{idata}). Section \ref{rei-se} is devoted to the derivation of the relative entropy inequality through the use of the
so-called ``augmented velocity'' version of the system (\ref{cont_cap}) - (\ref{mom_cap}).
In Section \ref{app} we perform the weak-strong uniqueness analysis for the system with and without the presence of the pressure term.  Section \ref{HM} is devoted to the proof of the high Mach number limit by means of the relative entropy.

\section{Preliminaries} \label{pre}

\subsection{Notation}
\noindent
We denote by $C_c^\infty([0,T)\times\mathbb{T}^3;\mathbb{R}^d)$ the space of periodic smooth functions with values in $\mathbb{R}^d$ with compact support in $[0,T)\times\mathbb{T}^3$ and by $L^p(\mathbb{T}^3)$ the standard Lebesgue spaces. The Sobolev spaces of functions with $k$ distributional derivative in $L^p(\mathbb{T}^3)$ are $W^{k,p}(\mathbb{T}^3)$. In the case $p=2$, $W^{k,p}(\mathbb{T}^3) = H^k(\mathbb{T}^3)$. The Bochner spaces for time-dependent functions with values in Banach spaces $X$ are denoted by $L^p(0,T;X)$ and $W^{k,p}(0,T;X)$. The space $C(0,T;X_w)$ is the space of continuous functions endowed with the weak topology. The quantities $Du$ and $Au$ are the symmetric and anti-symmetric part of the gradient $\nabla u$, respectively.
As a matter of notation, we always write $(\nabla u)^t := \nabla^t u$, meaning the transpose matrix. The above notation holds both for the velocity field $u$ and for the other velocity fields that we will introduce in the analysis. 

\subsection{Existence of weak solutions}
\noindent
 In \cite{BrDeLi} the authors prove the existence of arbitrarily large, global-in-time finite energy weak solutions to (\ref{cont_cap}) - (\ref{mom_cap}) by considering test functions of the form $\varrho \varphi$, with $\varphi$ smooth and compactly supported. As stressed in \cite{AnSp}, this is somehow equivalent to considering test functions that are supported where the mass density is positive. In \cite{AnSp}, the authors improve the result in \cite{BrDeLi} by removing the requirement on the test functions and by considering a more natural definition of weak solutions (see Definition 2.1 in \cite{AnSp}). The result is obtained by considering a suitable approximate system and by the use of truncation arguments in order to obtain the sufficient convergence towards global-in-time finite energy weak solutions to (\ref{cont_cap}) - (\ref{idata}).

In the spirit of \cite{AnSp}, we introduce the definition of a weak solution for the system (\ref{cont_cap}) - (\ref{idata}):

\begin{definition} \label{weak-def}
A triple $(\varrho,u,\mathcal{T})$, with $\varrho \geq 0$, is said to be a weak solution to (\ref{cont_cap}) - (\ref{mom_cap}) with initial data (\ref{idata}) if the following conditions are satisfied:

\bigskip
(1) Integrability conditions:
\begin{equation*}
    \varrho \in L^\infty(0,T;H^1(\mathbb{T}^3))\cap
    L^2(0,T;L^2(\mathbb{T}^3)), \ \ \ \ 
    \sqrt{\varrho} u \in L^\infty(0,T;L^2(\mathbb{T}^3)),
\end{equation*}
\begin{equation*}
    \varrho^{\frac{\gamma}{2}}\in
    L^\infty(0,T;L^2(\mathbb{T}^3))\cap
    L^2(0,T;H^1(\mathbb{T}^3)), \ \ \ \ 
    \nabla \sqrt{\varrho} \in
    L^\infty(0,T;L^2(\mathbb{T}^3)),
\end{equation*}
\begin{equation*}
    \mathcal{T} \in 
    L^2(0,T;L^2(\mathbb{T}^3)), \ \ \ \ 
    \varrho u \in C([0,T); L_w^{\frac{3}{2}}(\mathbb{T}^3)).
\end{equation*}

\bigskip
(2) Continuity equation:

For any $\varphi \in C_c^\infty([0,T)\times\mathbb{T}^3;\mathbb{R})$,
\begin{equation} \label{cont-wf}
\int_{\mathbb{T}^3}
\varrho_0 \varphi_0(x) dx + \int_0^T \int_{\mathbb{T}^3}
\varrho \varphi_t + \sqrt{\varrho}\sqrt{\varrho} u \nabla \varphi dx dt = 0.
\end{equation}

\bigskip
(3) Momentum equation:

For any fixed $l=1,2,3$ and $\phi \in C_c^\infty([0,T)\times\mathbb{T}^3;\mathbb{R}^3),$

\begin{equation*} 
    \int_{\mathbb{T}^3} \varrho_0 u_0 \phi(0) dx 
    + \int_0^T \int_{\mathbb{T}^3}
    \sqrt{\varrho}\left(\sqrt{\varrho}u^l\right) \phi_t \; dx dt
    + \int_0^T \int_{\mathbb{T}^3}
    \sqrt{\varrho}u^l \sqrt{\varrho}u \cdot \nabla \phi \; dx dt
\end{equation*}
\begin{equation*}
    - 2\nu \int_0^T \int_{\mathbb{T}^3} 
    \sqrt{\varrho} \mathcal{S}_{\cdot,l}
    \nabla \phi \; dx dt
    - 2 \int_0^T \int_{\mathbb{T}^3}
    \nabla \varrho^{\frac{\gamma}{2}} \varrho^{\frac{\gamma}{2}} \cdot \phi \; dx dt
    - 2\kappa^2 \int_0^T \int_{\mathbb{T}^3} 
    \nabla_l \varrho \Delta \varrho \phi \; dx dt
\end{equation*}
\begin{equation} \label{mom-wf}
    - 2\kappa^2 \int_0^T \int_{\mathbb{T}^3}
    \varrho \Delta \varrho \nabla_l \phi \; dx dt = 0.
\end{equation}

\bigskip
(4) Dissipation:

For any $\xi \in C_c^\infty([0,T)\times\mathbb{T}^3;\mathbb{R})$,

\begin{equation} \label{diss}
    \int_0^T \int_{\mathbb{T}^3}
    \sqrt{\varrho} \mathcal{T}_{i,j} \xi \; dxdt
    =
    - \int_0^T \int_{\mathbb{T}^3}
    \varrho u_i \nabla_j \xi \; dxdt
    - \int_0^T \int_{\mathbb{T}^3} 
    2 \sqrt{\varrho} u_i \otimes \nabla_j \sqrt{\varrho} \xi \; dxdt.
\end{equation}

\bigskip
(5) Energy inequality:

The following energy inequality holds:
\begin{equation*} 
    \sup_{t \in (0,T)} \int_{\mathbb{T}^3}
    \left[
    \frac{\varrho(t,x)|u(t,x)|^2}{2} + H(\varrho)
    + \frac{|\nabla \varrho(t,x)|^2}{2}
    \right] dx
    + 2\nu \int_0^T \int_{\mathbb{T}^3}
    |\mathcal{S}(u)(t,x)|^2 dx
\end{equation*}
\begin{equation} \label{ei-wf}
    \leq \int_{\mathbb{T}^3} 
    \left[ \varrho_0(x) |u_0(x)|^2 + H(\varrho_0) + \frac{|\nabla \varrho_0(x)|^2}{2} \right]dx.  
\end{equation}
\end{definition}

\noindent
Here, $H(\varrho)$ is such that
\begin{equation}
H''(\varrho) = \frac{p'(\varrho)}{\varrho}
\label{hpressure}
\end{equation}
and $\mathcal{S}(u)$ is the symmetric part of the tensor $\mathcal{T}(u)$ defined by
\begin{equation} \label{TT}
    \sqrt{\varrho}\mathcal{T}^{jk}(u)
    = \partial_j (\varrho u_k) - 2 \partial_j \sqrt{\varrho} (\sqrt{\varrho} u_k).
\end{equation}

\begin{remark}
As stressed in \cite{AnSp}, for smooth solutions for the system (\ref{cont_cap}) - (\ref{mom_cap}) the energy inequality (in fact an equality) reads 
\begin{equation} \label{ee-sm}
    E(T) + \int_0^T \int_{\mathbb{T}^3} \varrho|Du|^2 dxdt \leq E(0)
\end{equation}
where
\begin{equation*}
    E(T) = \int_{\mathbb{T}^3}
    \left[
    \frac{\varrho|u|^2}{2}
    + H(\varrho) + 
    \frac{|\nabla \varrho|^2}{2}
    \right] dx.
\end{equation*}
Indeed, it is not clear whether arbitrary finite energy weak solutions satisfy inequality (\ref{ee-sm}) because we are not able to conclude that the weak-limit in $L_{t,x}^2$ of $\sqrt{\varrho_n}Du_n$ is $\sqrt{\varrho}Du$. In general, we can only say that 
\begin{equation*}
    \sqrt{\varrho_n}Du_n \rightharpoonup \mathcal{S} \ \ \text{in} \ \ L_{t,x}^2.
\end{equation*}
hence, the viscous term $\sqrt{\varrho} Du$ has to be understood as $\sqrt{\varrho}\mathcal{S}$.
\end{remark}

\noindent
In \cite{AnSp}, the following existence result has been proved.
\begin{theorem} \label{main-ex}
Assume $\varrho_0$ and $\varrho_0 u_0$ satisfy
\begin{equation} \label{id-th-1}
    \varrho_0 \geq 0, \ \ \ \ \varrho_0 \in L^1 \cap L^\gamma (\mathbb{T}^3), \ \ \ \ 
    \nabla \sqrt{\varrho_0} \in L^2 (\mathbb{T}^3), \ \ \ \ \log \varrho_0 \in L^1(\mathbb{T}^3),
\end{equation}
\begin{equation} \label{id-th-2}
    \sqrt{\varrho_0} u_0 \in L^2(\mathbb{T}^3), \ \ \ \ 
    \varrho_0 u_0 \in L^p(\mathbb{T}^3) \ \ \text{with} \ \ p<2.
\end{equation}
Then, there exists at least a global weak solution $(\varrho, u, \mathcal{T})$ of (\ref{cont_cap}) - (\ref{mom_cap}) in the sense of Definition \ref{weak-def}. 
\end{theorem}

\begin{remark}
In the case of the pressureless system  (\ref{cont_cap-press}) - (\ref{mom_cap-press}) an existence results for global in time weak solutions in the framework of  \cite{BrDeLi} has been proved in \cite{CaDo}, see Theorem 2.4. 
Nevertheless it is possible to show that Theorem \ref{main-ex} holds also for the pressureless system (\ref{cont_cap-press}) - (\ref{mom_cap-press}). The main difference in the two results lies in the definition of weak solutions. Indeed in \cite{CaDo}, the used test functions 
are of the form $\varrho\varphi$ basically they are supported on the sets of positive density. In this paper we will work with weak solutions in the sense of  the Definition \ref{weak-def} suitably modified for the pressureless system.
\end{remark}

\section{Relative entropy inequality} \label{rei-se}
\noindent
This section is devoted to the study of a relative entropy inequality. For density dependent viscosity fluids this analysis requires the introduction of an ``augmented velocity'' system due to the fact that an $H^1$ bound for the velocity is no longer available because of the density dependent viscosity. Consequently, standard application of the Korn's inequality in the usual weak-strong uniqueness framework is not possible.

A relative entropy inequality for compressible Navier-Stokes equations with density-dependent viscosity has been introduced in \cite{BNV-1} and \cite{BNV-2} in the framework of $k$-entropy solutions (see \cite{BrDeZa}). 
In \cite{BNV-2}, the authors study a weak-strong uniqueness property together with some applications.
Later, in \cite{BGL}, the authors consider weak solutions for quantum fluids and a weak-strong uniqueness analysis is presented. Differently from \cite{BNV-1}, \cite{BNV-2}, they define a relative entropy functional with the presence of viscous terms (see Section 4.1 in \cite{BGL}).

Differently from the above analysis, we define a relative entropy functional for the ``augmented velocity'' system that contains the energy contribution coming from the capillarity. Indeed, while in the quantum case, the energy contribution coming from the quantum term could be absorbed in the ``augmented velocity'' and is not present in the definition of the energy functional (see, for example, Section 3 in \cite{CaDo}), this is not the case for capillary fluids.
Moreover, in the same spirit of \cite{BNV-1} and \cite{BNV-2}, the viscous terms are not present in our definition of the relative entropy functional.


\subsection{``Augmented" system}

\noindent
In order to derive an augmented version of the system (\ref{cont_cap}) - (\ref{mom_cap}), we set the following notations for a general $\mu(\varrho)$ that will be specified later:
$$
\widetilde{\phi}'(\varrho)=\frac{\mu'(\varrho)}{\varrho}, \ \ \ \ \nabla \mu(\varrho) = \varrho \nabla \widetilde{\phi}(\varrho).
$$
\noindent
We multiply the continuity equation by $\mu'(\varrho)$ and we write the corresponding equation for $\mu(\varrho)$. We have,
\begin{equation} \label{cont-mu}
    \partial_t \mu(\varrho) + \text{div}(\mu(\varrho)u)=0. 
\end{equation}
We differentiate (\ref{cont-mu}) respect to space and observing that $\nabla \mu(\varrho) = \varrho \nabla \widetilde{\phi}(\varrho)$, we obtain
\begin{equation} \label{cont-phi}
    \partial_t (\varrho \nabla \widetilde{\phi} (\varrho)) + \text{div}(\varrho u \otimes \nabla \widetilde{\phi}(\varrho))
    + \text{div}(\mu(\varrho)\nabla^t u)=0. 
\end{equation}
In the same spirit of \cite{BrDeZa}, we define the so called ``effective velocity''
\begin{equation*}
    w = u + \nabla \widetilde{\phi} (\varrho).
\end{equation*}
From the momentum equation (\ref{mom_cap}) and \eqref{cont-phi}, we get
$$
\partial_t (\varrho w) + \textrm{div} (\varrho w \otimes u) + \nabla p(\varrho) - \textrm{div} (\mu(\varrho) Du)
- \textrm{div} (\mu(\varrho) Au)
$$
\begin{equation} \label{mom_cap_2}
    - 2\kappa^2 \left(
    \nabla (\varrho \Delta \varrho) + \frac{1}{2} \nabla ( |\nabla \varrho|^2)
    - 4 \textrm{div}(\varrho \nabla \sqrt{\varrho} \otimes \nabla \sqrt{\varrho})
    \right)=0.
\end{equation}
The relation above could be rewritten as follows
$$
\partial_t (\varrho w) + \textrm{div} (\varrho w \otimes u) + \nabla p(\varrho) - \textrm{div} (\mu(\varrho) Dw)
- \textrm{div} (\mu(\varrho) Aw)
+ \textrm{div} (\mu(\varrho) \nabla \nabla \widetilde{\phi}(\varrho))
$$
\begin{equation} \label{mom_cap_3}
    - 2\kappa^2 \left(
    \nabla (\varrho \Delta \varrho) + \frac{1}{2} \nabla ( |\nabla \varrho|^2)
    - 4 \textrm{div}(\varrho \nabla \sqrt{\varrho} \otimes \nabla \sqrt{\varrho})
    \right)=0.
\end{equation}
Now, for $\mu(\varrho) = 2 \nu \varrho$ we define $v = \nabla \widetilde{\phi} (\varrho)=2\nu\nabla\log \varrho$ and we obtain the ``augmented velocity''  version of the Navier-Stokes-Korteweg system
\begin{equation} \label{ag_1.0}
    \partial_t \varrho + \textrm{div} (\varrho u) = 0,
\end{equation}
$$
\partial_t (\varrho w) + \textrm{div}(\varrho w \otimes u) + \nabla p(\varrho)
    - 2\nu \textrm{div}(\varrho \nabla w) + 2\nu \textrm{div} (\varrho \nabla v)
$$
\begin{equation} \label{ag_2.0}
    - 2\kappa^2 \left(
    \nabla (\varrho \Delta \varrho) + \frac{1}{2} \nabla ( |\nabla \varrho|^2)
    - 4 \textrm{div}(\varrho \nabla \sqrt{\varrho} \otimes \nabla \sqrt{\varrho})
    \right)=0,
\end{equation}
\begin{equation} \label{ag_3.0}
    \partial_t (\varrho v) + \textrm{div} (\varrho v \otimes u) + 2\nu \textrm{div} (\varrho \nabla^t u) = 0.
\end{equation}

\begin{definition} \label{wf-aug}
    We say that $(\varrho,v,w)$ is a weak solution of the ``augmented velocity'' system if (\ref{ag_1.0}) - (\ref{ag_3.0}) are satisfied in the distribution sense, namely
\begin{equation} \label{weak-cont}
    - \int_{\mathbb{T}^3} (\varrho \cdot \varphi)(T,\cdot)dx  
    + \int_{\mathbb{T}^3} (\varrho \cdot \varphi)(0,\cdot)dx
    + \int_0^T \int_{\mathbb{T}^3} \left( \varrho \partial_t \varphi + \varrho u \cdot \nabla \varphi \right) dxdt=0,
\end{equation}
for any $\varphi \in C_c^\infty([0,T)\times\mathbb{T}^3;\mathbb{R})$.
$$
- \int_{\mathbb{T}^3} (\varrho w \cdot \phi)(T,\cdot)dx  
+ \int_{\mathbb{T}^3} (\varrho w \cdot \phi)(0,\cdot)dx
$$
$$
+ \int_0^T \int_{\mathbb{T}^3} 
\left(
\varrho w \cdot \partial_t \phi + \varrho w \otimes u : \nabla \phi + p(\varrho) \textrm{div}\phi
- 2\nu \varrho \nabla w:\nabla\phi + 2\nu \varrho \nabla v:\nabla \phi
\right)
dxdt
$$
\begin{equation} \label{weak-w}
    +2\kappa^2 \int_0^T \int_{\mathbb{T}^3} 
    \left( 
    \nabla\varrho \cdot \nabla( \varrho \textrm{div}\phi) 
    - \frac{1}{2} |\nabla \varrho|^2 \textrm{div}\phi
    + \nabla \varrho \otimes \nabla \varrho : \nabla \phi
    \right)dxdt=0,
\end{equation}
for any $\phi \in C_c^\infty([0,T)\times\mathbb{T}^3;\mathbb{R}^{3})$
$$
- \int_{\mathbb{T}^3} (\varrho v \cdot \varphi)(T,\cdot)dx  
+ \int_{\mathbb{T}^3} (\varrho v \cdot \varphi)(0,\cdot)dx.
$$
\begin{equation} \label{weak-v-bar}
+ \int_0^T \int_{\mathbb{T}^3}
\left(
\varrho v \cdot\partial_t\varphi 
+  \varrho v \otimes u : \nabla\varphi 
- 2\nu \varrho \nabla^t  v:\nabla\varphi
+ 2\nu \varrho \nabla^t w: \nabla\varphi
\right)dxdt
= 0,
\end{equation}
for any $\varphi \in C_c^\infty([0,T)\times\mathbb{T}^3;\mathbb{R}^3)$.\\
Moreover, the following energy inequality holds
  \begin{equation}
   \label{ei}
  \begin{split}
    \frac{d}{dt} \int_{\mathbb{T}^3} \Big( \frac{1}{2} \varrho(t,x) \left( |w(t,x)|^2 + |v(t,x)|^2 \right) &+ H(\varrho) + \kappa^2|\nabla \varrho(t,x)|^2 \Big) dx\\
 +2\nu \int_{\mathbb{T}^3} \Big( |\mathcal{S}(u)(t,x)|^2 + |\mathcal{A}(w)(t,x)|^2   &+ \frac{p'(\varrho)}{\varrho} |\nabla \varrho(t,x)|^2 \Big)dx 
    \\
    +4\nu \kappa^2 \int_{\mathbb{T}^3} |\Delta \varrho(t,x)|^2  dx
    \leq 0,
   \end{split}
\end{equation}
where $\mathcal{A}(w)$ is the anti-symmetric part of the tensor $\mathcal{T}(w)$ defined by (\ref{TT}) with $w$ in the place of $u$.
\end{definition}

In the case of the {\em pressureless system} (\ref{cont_cap-press}), (\ref{mom_cap-press}), the ``augmented velocity''  system is given by 
\begin{equation} \label{agpl_1.0}
    \partial_t \varrho + \textrm{div} (\varrho u) = 0,
\end{equation}
$$
\partial_t (\varrho w) + \textrm{div}(\varrho w \otimes u) 
    - 2\nu \textrm{div}(\varrho \nabla w) + 2\nu \textrm{div} (\varrho \nabla v)
$$
\begin{equation} \label{agpl_2.0}
    - 2\kappa^2 \left(
    \nabla (\varrho \Delta \varrho) + \frac{1}{2} \nabla ( |\nabla \varrho|^2)
    - 4 \textrm{div}(\varrho \nabla \sqrt{\varrho} \otimes \nabla \sqrt{\varrho})
    \right)=0,
\end{equation}
\begin{equation} \label{agpl_3.0}
    \partial_t (\varrho v) + \textrm{div} (\varrho v \otimes u) + 2\nu \textrm{div} (\varrho \nabla^t u) = 0.
\end{equation}

\begin{definition} \label{wf-augpl}
    We say that $(\varrho,v,w)$ is a weak solution of the ``augmented velocity'' system if (\ref{agpl_1.0}) - (\ref{agpl_3.0}) are satisfied in the distribution sense, namely
\begin{equation} \label{weak-cont-pless}
    - \int_{\mathbb{T}^3} (\varrho \cdot \varphi)(T,\cdot)dx  
    + \int_{\mathbb{T}^3} (\varrho \cdot \varphi)(0,\cdot)dx
    + \int_0^T \int_{\mathbb{T}^3} \left( \varrho \partial_t \varphi + \varrho u \cdot \nabla \varphi \right) dxdt=0,
\end{equation}
for any $\varphi \in C_c^\infty([0,T)\times\mathbb{T}^3;\mathbb{R})$.
$$
- \int_{\mathbb{T}^3} (\varrho w \cdot \phi)(T,\cdot)dx  
+ \int_{\mathbb{T}^3} (\varrho w \cdot \phi)(0,\cdot)dx
$$
$$
+ \int_0^T \int_{\mathbb{T}^3} 
\left(
\varrho w \cdot \partial_t \phi + \varrho w \otimes u : \nabla \phi - 2\nu \varrho \nabla w:\nabla\phi + 2\nu \varrho \nabla v:\nabla \phi
\right)
dxdt
$$
\begin{equation} \label{weak-w-pless}
    +2\kappa^2 \int_0^T \int_{\mathbb{T}^3} 
    \left( 
    \nabla\varrho \cdot \nabla( \varrho \textrm{div}\phi) 
    - \frac{1}{2} |\nabla \varrho|^2 \textrm{div}\phi
    + \nabla \varrho \otimes \nabla \varrho : \nabla \phi
    \right)dxdt=0,
\end{equation}
for any $\phi \in C_c^\infty([0,T)\times\mathbb{T}^3;\mathbb{R}^{3})$
$$
- \int_{\mathbb{T}^3} (\varrho v \cdot \varphi)(T,\cdot)dx  
+ \int_{\mathbb{T}^3} (\varrho v \cdot \varphi)(0,\cdot)dx.
$$
\begin{equation} \label{weak-v-bar-pless}
+ \int_0^T \int_{\mathbb{T}^3}
\left(
\varrho v \cdot\partial_t\varphi 
+  \varrho v \otimes u : \nabla\varphi 
- 2\nu \varrho \nabla^t  v:\nabla\varphi
+ 2\nu \varrho \nabla^t w: \nabla\varphi
\right)dxdt
= 0,
\end{equation}
for any $\varphi \in C_c^\infty([0,T)\times\mathbb{T}^3;\mathbb{R}^3)$.\\
Moreover, the following energy inequality holds
 $$
    \frac{d}{dt} \int_{\mathbb{T}^3} \left( \frac{1}{2} \varrho(t,x) \left( |w(t,x)|^2 + |v(t,x)|^2 \right) + \kappa^2|\nabla \varrho(t,x)|^2 \right) dx
    $$
\begin{equation} \label{ei-pless}
    +2\nu \int_{\mathbb{T}^3} \left( |\mathcal{S}(u)(t,x)|^2 + |\mathcal{A}(w)(t,x)|^2   \right)dx 
    +4\nu \kappa^2 \int_{\mathbb{T}^3} |\Delta \varrho(t,x)|^2  dx
    \leq 0.
\end{equation}
\end{definition}

\begin{remark} \label{visc-und}
The viscous terms present in the relations (\ref{weak-w-pless}) and (\ref{weak-v-bar-pless}) should be understood as in (\ref{diss}), namely in terms of the symmetric and anti-symmetric part of the tensor $\mathcal{T}$ defined by (\ref{TT}). 
\end{remark}

\begin{remark}
As the authors remarked in \cite{BGL}, a global weak solution $(\varrho, u)$ of the compressible Navier-Stokes Korteweg system is also a solution of the ``augmented velocity''  version. Consequently, Theorem \ref{main-ex} holds for weak solutions to the system
\eqref{ag_1.0} - \eqref{ag_3.0} in the sense of Definition \ref{wf-aug} and we have the equivalent results in the pressureless case.
\end{remark}


\subsection{Relative entropy inequality for  capillary Navier-Stokes fluids}
\label{SRE}
\noindent
In this section we define a relative entropy inequality for the relative entropy functional given by
$$
\mathcal{E}(T,\cdot) = \mathcal{E}(\varrho, w, v|r, W, V)(T,\cdot)$$
$$=  \int_{\mathbb{T}^3} 
\left[
\frac{1}{2} \varrho \left( |w-W|^2 + |v-V|^2 \right) 
+ \kappa^2 |\nabla \varrho - \nabla r|^2  
\right]
dx 
+ \int_{\mathbb{T}^3} H(\varrho|r) dx
$$
where
$$
H(\varrho|r) = H(\varrho) - H(r) - H'(r)(\varrho - r)
$$
and $(\varrho, v, w)$ are weak solutions of the augmented system (\ref{ag_1.0}) - (\ref{ag_3.0}) and $(r,V, W)$ are smooth enough states of the fluid. Our result reads as follows.

\begin{theorem} \label{rei-th}
    Let $(\varrho,v,w)$ be a weak solution of the ``augmented velocity'' system (\ref{ag_1.0}) - (\ref{ag_3.0}) in the sense of Definition \ref{wf-aug}. For any smooth functions $(r,W,V)$, we have the following relative energy inequality
    $$
\mathcal{E}(T,\cdot) - \mathcal{E}(0,\cdot)
\leq
\int_0^T \int_{\mathbb{T}^3} \varrho \left( \partial_t V \cdot \left( V - v \right)
+
\left( \nabla V u \right) \cdot \left( V - v \right) \right)
dxdt
$$
$$
+ \int_0^T \int_{\mathbb{T}^3} \varrho \left( \partial_t W \cdot \left( W - w \right)
+
\left( \nabla W u \right) \cdot \left( W - w \right) \right)
dxdt
$$
$$
+ 2\kappa^2 \int_0^T \int_{\mathbb{T}^3} \left( \varrho \partial_t \Delta r + \varrho u \cdot \nabla \Delta r - \partial_t r \Delta r\right)
dxdt
$$
$$
- 2\kappa^2 \int_0^T \int_{\mathbb{T}^3} 
\left( 
\nabla\varrho \cdot \nabla( \varrho \textrm{div} W) 
- \frac{1}{2} |\nabla \varrho|^2 \textrm{div} W
+ \nabla \varrho \otimes \nabla \varrho : \nabla W
\right)
dxdt
$$
$$
+ 2\nu \int_0^T \int_{\mathbb{T}^3} \varrho \left( \nabla^t v : \nabla V + \nabla w: \nabla W  \right)
dxdt
- 2\nu \int_0^T \int_{\mathbb{T}^3} \varrho
\left(
\nabla v:\nabla W
+  \nabla^t w: \nabla
V
\right)
dxdt
$$
$$
- \int_0^T \int_{\mathbb{T}^3} 
\left[
\partial_t (H'(r))(\varrho - r)
+ \varrho u \nabla (H'(r)) + p(\varrho) \textrm{div} W
\right]dxdt
$$
\begin{equation} \label{rei-res}
-2\nu \int_0^T \int_{\mathbb{T}^3} \left( |\mathcal{S}(u)|^2 + |\mathcal{A}(w)|^2   + \frac{p'(\varrho)}{\varrho} |\nabla \varrho|^2 \right)dxdt 
-4\nu \kappa^2 \int_0^T \int_{\mathbb{T}^3} |\Delta \varrho|^2  dxdt
\end{equation}
\end{theorem}

\begin{remark} \label{visc-rei}
As in Definition \ref{wf-augpl}, the viscous terms present in the right hand side of (\ref{rei-res}), namely
$$
2\nu \int_0^T \int_{\mathbb{T}^3} \varrho \left( \nabla^t v : \nabla V + \nabla w: \nabla W  \right)
dxdt
- 2\nu \int_0^T \int_{\mathbb{T}^3} \varrho
\left(
\nabla v:\nabla W
+  \nabla^t w: \nabla
V
\right)
dxdt
$$
should be interpreted in terms of the symmetric and anti-symmetric part of the tensor $\mathcal{T}$ defined by (\ref{TT}). 
\end{remark}

\begin{proof}
\noindent
Thanks to the energy inequality (\ref{ei}), we can write
$$
\mathcal{E}(T,\cdot) - \mathcal{E}(0,\cdot) \leq 
\int_{\mathbb{T}^3} \left( \frac{1}{2} \varrho |W|^2 - \varrho w \cdot W + \frac{1}{2}\varrho |V|^2 - \varrho v \cdot V
+\kappa^2|\nabla r|^2 - 2\kappa^2\nabla \varrho \nabla r\right) (T,\cdot)
dx
$$
$$
- \int_{\mathbb{T}^3} \left( \frac{1}{2} \varrho |W|^2 - \varrho w \cdot W + \frac{1}{2}\varrho |V|^2 - \varrho v \cdot
V
+\kappa^2|\nabla r|^2 - 2\kappa^2\nabla \varrho \nabla r\right) (0,\cdot)
dx
$$
$$
- \int_{\mathbb{T}^3} \left( H(r) + H'(r)(\varrho-r) \right) (T,\cdot)
dx
+ \int_{\mathbb{T}^3} \left( H(r) + H'(r)(\varrho-r) \right) (0,\cdot)
dx
$$
\begin{equation} \label{step-1}
-2\nu \int_0^T\int_{\mathbb{T}^3} \left( |\mathcal{S}(u)|^2 + |\mathcal{A}(w)|^2   + \frac{p'(\varrho)}{\varrho} |\nabla \varrho|^2 \right)dx dt
-4\nu \kappa^2 \int_0^T\int_{\mathbb{T}^3} |\Delta \varrho|^2  dxdt.
\end{equation}
First, we test the continuity equation (\ref{weak-cont}) by $\displaystyle{\frac{1}{2}|W|^2}$, $\displaystyle{\frac{1}{2}|V|^2}$ and $2\kappa^2\Delta r$, respectively. We have,
\begin{align}
 - \int_{\mathbb{T}^3} \frac{1}{2} (\varrho |W|^2)(T,\cdot)dx  
    &+  \int_{\mathbb{T}^3} \frac{1}{2}(\varrho |W|^2)(0,\cdot)dx\notag\\
    &+ \int_0^T\!\!\! \int_{\mathbb{T}^3} \frac{1}{2} \left( \varrho \partial_t |W|^2 + \varrho u \cdot \nabla |W|^2 \right)
    dxdt=0, \label{test-w2}
    \end{align}
\begin{equation} \label{test-v2}
    - \int_{\mathbb{T}^3} \frac{1}{2} (\varrho |V|^2)(t)  
    +  \int_{\mathbb{T}^3} \frac{1}{2}(\varrho |V|^2)(0)
    + \int_0^T \int_{\mathbb{T}^3} \frac{1}{2} \left( \varrho \partial_t |V|^2 + \varrho u \cdot \nabla |V|^2 \right)
    dxdt=0,
\end{equation}
\begin{align} 
    2\kappa^2 \int_{\mathbb{T}^3} (\nabla \varrho \cdot \nabla r)(T,\cdot)dx
    &- 2\kappa^2 \int_{\mathbb{T}^3} (\nabla \varrho \cdot \nabla r)(0,\cdot)dx\notag\\
    &+ 2\kappa^2 \int_0^T \int_{\mathbb{T}^3} \left( \varrho \partial_t \Delta r + \varrho u \cdot \nabla \Delta r \right)dxdt =0.
\label{test-Dr}
\end{align}
Moreover, we have
\begin{equation} \label{grad-r2}
    \kappa^2\int_{\mathbb{T}^3} |\nabla r|^2 (T,\cdot)dx
    - \kappa^2 \int_{\mathbb{T}^3} |\nabla r|^2 (0,\cdot)dx 
    = - 2\kappa^2 \int_0^T \int_{\mathbb{T}^3} \partial_t r \Delta r dxdt.
\end{equation}
Now, we test the equation (\ref{weak-w}) by $W$, 
$$
- \int_{\mathbb{T}^3} (\varrho w \cdot W)(T,\cdot)dx
+ \int_{\mathbb{T}^3} (\varrho w \cdot W)(0,\cdot)
dx
$$
$$
+ \int_0^T \int_{\mathbb{T}^3} 
\left(
\varrho w \cdot \partial_t W + \varrho w \otimes u : \nabla W + p(\varrho) \textrm{div} W
- 2\nu \varrho \nabla w:\nabla W + 2\nu \varrho \nabla v:\nabla W
\right)dxdt    
$$
\begin{equation} \label{test-w}
    +2\kappa^2 \int_0^T \int_{\mathbb{T}^3}
    \left( 
    \nabla\varrho \cdot \nabla( \varrho \textrm{div} W) 
    - \frac{1}{2} |\nabla \varrho|^2 \textrm{div} W
    + \nabla \varrho \otimes \nabla \varrho : \nabla W
    \right)dxdt=0.
\end{equation}
and we test the equation (\ref{weak-v-bar}) by $V$,
$$
- \int_{\mathbb{T}^3} (\varrho v \cdot V)(T,\cdot)dx 
+ \int_{\mathbb{T}^3} (\varrho v \cdot V)(0,\cdot)dx
$$
\begin{equation} \label{test-v-bar}
+\int_0^T \int_{\mathbb{T}^3}
\left(
\varrho v\cdot\partial_t V
+  \varrho v \otimes u : \nabla V
- 2\nu \varrho \nabla^t v:\nabla
V
+ 2\nu \varrho \nabla^t w: \nabla
V
\right)
dxdt
= 0.
\end{equation}
Finally, using the continuity equation (\ref{cont_cap}), we have
$$
\int_0^T \int_{\mathbb{T}^3} 
\left[
\partial_t \left( H(r) + H'(r)(\varrho - r)
\right)
\right]
dxdt
$$
$$
= \int_0^T \int_{\mathbb{T}^3} 
\left[
H'(r)\partial_t r + \partial_t (H'(r))(\varrho - r) + H'(r)\partial_t \varrho - H'(r) \partial_t r
\right]dxdt
$$
$$
= \int_0^T \int_{\mathbb{T}^3} 
\left[
\partial_t (H'(r))(\varrho - r)
- H'(r)\textrm{div}(\varrho u)
\right]dxdt
$$
\begin{equation} \label{test-H}
    = \int_0^T \int_{\mathbb{T}^3} 
    \left[
    \partial_t (H'(r))(\varrho - r)
    + \varrho u \nabla (H'(r))
    \right]dxdt.
\end{equation}
Plugging (\ref{test-w2}) - (\ref{test-H}) in (\ref{step-1}), we obtain
$$
\mathcal{E}(T,\cdot) - \mathcal{E}(0,\cdot)
\leq
\int_0^T \int_{\mathbb{T}^3} \frac{1}{2} \left( \varrho \partial_t |W|^2 + \varrho u \cdot \nabla |W|^2 \right)dxdt
+ \int_0^T \int_{\mathbb{T}^3} \frac{1}{2} \left( \varrho \partial_t |V|^2 + \varrho u \cdot \nabla |V|^2 \right)dxdt
$$
$$
+ 2\kappa^2 \int_0^T \int_{\mathbb{T}^3} \left( \varrho \partial_t \Delta r + \varrho u \cdot \nabla \Delta r - \partial_t r \Delta r\right)dxdt
$$
$$
- \int_0^T \int_{\mathbb{T}^3} 
\left(
\varrho w \cdot \partial_t W + \varrho w \otimes u : \nabla W + p(\varrho) \textrm{div} W
- 2\nu \varrho \nabla w:\nabla W + 2\nu \varrho \nabla v:\nabla W
\right)dxdt
$$
$$
- 2\kappa^2 \int_0^T \int_{\mathbb{T}^3} 
\left( 
\nabla\varrho \cdot \nabla( \varrho \textrm{div} W) 
- \frac{1}{2} |\nabla \varrho|^2 \textrm{div} W
+ \nabla \varrho \otimes \nabla \varrho : \nabla W
\right)dxdt
$$
$$
- \int_0^T \int_{\mathbb{T}^3}
\left(
\varrho v \cdot\partial_t V
+  \varrho v \otimes u : \nabla V
- 2\nu \varrho \nabla^t v:\nabla
V
+ 2\nu \varrho \nabla^t w: \nabla
V
\right)dxdt
$$
$$
- \int_0^T \int_{\mathbb{T}^3} 
\left[
\partial_t (H'(r))(\varrho - r)
+ \varrho u \nabla (H'(r))
\right]dxdt
$$
\begin{equation} \label{step-2}
-2\nu \int_0^T \int_{\mathbb{T}^3} \left( |\mathcal{S}(u)|^2 + |\mathcal{A}(w)|^2   + \frac{p'(\varrho)}{\varrho} |\nabla \varrho|^2 \right)dxdt 
-4\nu \kappa^2 \int_0^T \int_{\mathbb{T}^3} |\Delta \varrho|^2  dxdt.
\end{equation}
Rearranging (\ref{step-2}), we obtain
$$
\mathcal{E}(T,\cdot) - \mathcal{E}(0,\cdot)
\leq
\int_0^T \int_{\mathbb{T}^3} \varrho \left( \partial_t V \cdot \left( V - v \right)
+
\left( \nabla V u \right) \cdot \left( V - v \right) \right)
dxdt
$$
$$
+ \int_0^T \int_{\mathbb{T}^3} \varrho \left( \partial_t W \cdot \left( W - w \right)
+
\left( \nabla W u \right) \cdot \left( W - w \right) \right)
dxdt
$$
$$
+ 2\kappa^2 \int_0^T \int_{\mathbb{T}^3} \left( \varrho \partial_t \Delta r + \varrho u \cdot \nabla \Delta r - \partial_t r \Delta r\right)
dxdt
$$
$$
- 2\kappa^2 \int_0^T \int_{\mathbb{T}^3} 
\left( 
\nabla\varrho \cdot \nabla( \varrho \textrm{div} W) 
- \frac{1}{2} |\nabla \varrho|^2 \textrm{div} W
+ \nabla \varrho \otimes \nabla \varrho : \nabla W
\right)
dxdt
$$
$$
+ 2\nu \int_0^T \int_{\mathbb{T}^3} \varrho \left( \nabla^t v : \nabla V + \nabla w: \nabla W  \right)
dxdt
- 2\nu \int_0^T \int_{\mathbb{T}^3} \varrho
\left(
\nabla v:\nabla W
+  \nabla^t w: \nabla
V
\right)
dxdt
$$
$$
- \int_0^T \int_{\mathbb{T}^3} 
\left[
\partial_t (H'(r))(\varrho - r)
+ \varrho u \nabla (H'(r)) + p(\varrho) \textrm{div} W
\right]dxdt
$$
\begin{equation} \label{step-3}
-2\nu \int_0^T \int_{\mathbb{T}^3} \left( |\mathcal{S}(u)|^2 + |\mathcal{A}(w)|^2   + \frac{p'(\varrho)}{\varrho} |\nabla \varrho|^2 \right)dxdt 
-4\nu \kappa^2 \int_0^T \int_{\mathbb{T}^3} |\Delta \varrho|^2  dxdt
\end{equation}
that is exactly (\ref{rei-res}). 
\end{proof}

\subsection{Relative entropy inequality for pressureless capillary  fluids}
\noindent
We introduce the following relative entropy functional related to the pressureless system (\ref{cont_cap-press}) - (\ref{mom_cap-press}):
$$
\mathcal{E}(T,\cdot) = \mathcal{E}(\varrho, w, v|r, W, V)(T,\cdot)
$$
\begin{equation} \label{ref-pless}
=  \int_{\mathbb{T}^3} 
\left[
\frac{1}{2} \varrho \left( |w-W|^2 + |v-V|^2 \right) 
+ \kappa^2 |\nabla \varrho - \nabla r|^2  
\right]
dx,
\end{equation}
where $(\varrho, v, w)$ are weak solutions of the augmented system (\ref{agpl_1.0}) - (\ref{agpl_3.0}) and $(r,V, W)$ are smooth enough states of the fluid.
Now, the idea is to derive a relative entropy inequality without the presence of the pressure terms and to apply the same argument of Section \ref{SRE}.

Thanks to the energy inequality (\ref{ei-pless}), we can write
$$
\mathcal{E}(T,\cdot) - \mathcal{E}(0,\cdot) \leq 
\int_{\mathbb{T}^3} \left( \frac{1}{2} \varrho |W|^2 - \varrho w \cdot W + \frac{1}{2}\varrho |V|^2 - \varrho v \cdot V
+\kappa^2|\nabla r|^2 - 2\kappa^2\nabla \varrho \nabla r\right) (T,\cdot)
dx
$$
$$
- \int_{\mathbb{T}^3} \left( \frac{1}{2} \varrho |W|^2 - \varrho w \cdot W + \frac{1}{2}\varrho |V|^2 - \varrho v \cdot
V
+\kappa^2|\nabla r|^2 - 2\kappa^2\nabla \varrho \nabla r\right) (0,\cdot)
dx
$$
\begin{equation} \label{step-1-pless}
-2\nu \int_0^T\int_{\mathbb{T}^3} \left( |\mathcal{S}(u)|^2 + |\mathcal{A}(w)|^2 \right)dx dt
-4\nu \kappa^2 \int_0^T\int_{\mathbb{T}^3} |\Delta \varrho|^2  dxdt.
\end{equation}
We test the continuity equation (\ref{weak-cont-pless}) by $\displaystyle{\frac{1}{2}|W|^2}$, $\displaystyle{\frac{1}{2}|V|^2}$ and $2\kappa^2\Delta r$, the weak formulations 
 (\ref{weak-w-pless}), (\ref{weak-v-bar-pless}) by $W$ and  $V$ respectively and plugging them in (\ref{step-1-pless}), we obtain
$$\mathcal{E}(T,\cdot) - \mathcal{E}(0,\cdot)
\leq
\int_0^T \int_{\mathbb{T}^3} \frac{1}{2} \left( \varrho \partial_t |W|^2 + \varrho u \cdot \nabla |W|^2 \right)dxdt
$$
$$
+ \int_0^T \int_{\mathbb{T}^3} \frac{1}{2} \left( \varrho \partial_t |V|^2 + \varrho u \cdot \nabla |V|^2 \right)dxdt
$$
$$
+ 2\kappa^2 \int_0^T \int_{\mathbb{T}^3} \left( \varrho \partial_t \Delta r + \varrho u \cdot \nabla \Delta r - \partial_t r \Delta r\right)dxdt
$$
$$
- \int_0^T \int_{\mathbb{T}^3} 
\left(
\varrho w \cdot \partial_t W + \varrho w \otimes u : \nabla W 
- 2\nu \varrho \nabla w:\nabla W + 2\nu \varrho \nabla v:\nabla W
\right)dxdt
$$
$$
- 2\kappa^2 \int_0^T \int_{\mathbb{T}^3} 
\left( 
\nabla\varrho \cdot \nabla( \varrho \textrm{div} W) 
- \frac{1}{2} |\nabla \varrho|^2 \textrm{div} W
+ \nabla \varrho \otimes \nabla \varrho : \nabla W
\right)dxdt
$$
$$
- \int_0^T \int_{\mathbb{T}^3}
\left(
\varrho v \cdot\partial_t V
+  \varrho v \otimes u : \nabla V
- 2\nu \varrho \nabla^t v:\nabla
V
+ 2\nu \varrho \nabla^t w: \nabla
V
\right)dxdt
$$
\begin{equation} \label{step-2-pless}
-2\nu \int_0^T \int_{\mathbb{T}^3} \left( |\mathcal{S}(u)|^2 + |\mathcal{A}(w)|^2   + \frac{p'(\varrho)}{\varrho} |\nabla \varrho|^2 \right)dxdt 
-4\nu \kappa^2 \int_0^T \int_{\mathbb{T}^3} |\Delta \varrho|^2  dxdt.
\end{equation}
Rearranging (\ref{step-2-pless}), we obtain the following relative entropy inequality for the pressureless case,
$$
\mathcal{E}(T,\cdot) - \mathcal{E}(0,\cdot)
\leq
\int_0^T \int_{\mathbb{T}^3} \varrho \left( \partial_t V \cdot \left( V - v \right)
+
\left( \nabla V u \right) \cdot \left( V - v \right) \right)
dxdt
$$
$$
+ \int_0^T \int_{\mathbb{T}^3} \varrho \left( \partial_t W \cdot \left( W - w \right)
+
\left( \nabla W u \right) \cdot \left( W - w \right) \right)
dxdt
$$
$$
+ 2\kappa^2 \int_0^T \int_{\mathbb{T}^3} \left( \varrho \partial_t \Delta r + \varrho u \cdot \nabla \Delta r - \partial_t r \Delta r\right)
dxdt
$$
$$
- 2\kappa^2 \int_0^T \int_{\mathbb{T}^3} 
\left( 
\nabla\varrho \cdot \nabla( \varrho \textrm{div} W) 
- \frac{1}{2} |\nabla \varrho|^2 \textrm{div} W
+ \nabla \varrho \otimes \nabla \varrho : \nabla W
\right)
dxdt
$$
$$
+ 2\nu \int_0^T \int_{\mathbb{T}^3} \varrho \left( \nabla^t v : \nabla V + \nabla w: \nabla W  \right)
dxdt
- 2\nu \int_0^T \int_{\mathbb{T}^3} \varrho
\left(
\nabla v:\nabla W
+  \nabla^t w: \nabla
V
\right)
dxdt
$$
\begin{equation} \label{step-3-pless}
-2\nu \int_0^T \int_{\mathbb{T}^3} \left( |\mathcal{S}(u)|^2 + |\mathcal{A}(w)|^2   \right)dxdt 
-4\nu \kappa^2 \int_0^T \int_{\mathbb{T}^3} |\Delta \varrho|^2  dxdt.
\end{equation}

\section{Application 1: Weak-Strong Uniqueness} \label{app}
\noindent
This section is devoted to the proof of  weak strong uniqueness properties which are  one of the main application of the relative entropy inequality (\ref{rei-res}). We will show these results proper both for the Navier Stokes capillary system (\ref{cont_cap}) - (\ref{mom_cap}) and for the capillary pressureless model (\ref{cont_cap-press}) - (\ref{mom_cap-press}). Moreover for this latter model we will analyse the particular case in which the initial data are given by a strictly positive initial density $\varrho_{0}>0$ and an irrotational initial velocity $u_{0}= -\nabla \widetilde{\phi}(\varrho_0)$.

\subsection{Weak-strong uniqueness for capillary fluids} \label{ws-un}
\noindent
We introduce the system satisfied by the strong solution $(r,W,V)$:
\begin{equation} \label{cont-r}
    \partial_t r + \textrm{div}(rU) =0,
\end{equation}
\begin{equation} \label{mom-W}
    r \left( \partial_t W + \nabla W \cdot U \right) + \nabla p(r) - 2\nu \textrm{div} (r \nabla W)
    + 2\nu \textrm{div} (r \nabla V)
    - 2\kappa^2
    r \nabla \Delta r =0,
\end{equation}
\begin{equation} \label{mom-V}
    r \left( \partial_t V + \nabla V \cdot U \right)
    - 2\nu \textrm{div} (r \nabla^t V)
    + 2\nu \textrm{div} (r \nabla^t W) = 0
\end{equation}
with 
$U=W-V$ and $V=2\nu\nabla\log r$ and $(r,W,V)$
belonging to the following regularity class
\begin{equation} \label{reg_class}
\begin{split}
0<\inf_{\left(0,T\right)\times\Omega}&r\leq r\leq\sup_{\left(0,T\right)\times\Omega}r<+\infty,\\
\nabla r\in L^{2}\left(0,T;L^{\infty}\left(\Omega\right)\right)&\cap L^{1}\left(0,T;W^{1,\infty}\left(\Omega\right)\right),\\
W\in L^{\infty}\left(0,T;W^{2,\infty}\left(\Omega\right)\right)&\cap W^{1,\infty}\left(0,T;L^{\infty}\left(\Omega\right)\right),\\
V\in L^{\infty}\left(0,T;W^{2,\infty}\left(\Omega\right)\right)&\cap W^{1,\infty}\left(0,T;L^{\infty}\left(\Omega\right)\right)\\
\partial_{t}H'(r)\in L^{1}(0,T;L^{\gamma/\gamma -1}(\Omega)), &\quad \nabla H'(r)\in L^{1}(0,T;L^{2\gamma/\gamma -1}(\Omega)).
\end{split}
\end{equation}
We remark that any strong solution of the system (\ref{cont_cap}) - (\ref{mom_cap}) is a strong solution of the above augmented system. Concerning the existence of strong solution for capillary fluids we recall that   the local existence of smooth solutions was proved  in  \cite{HaLi94}, while in \cite{HaLi96} the authors got the existence of global smooth solutions but for small initial data.

Our main result of this section is the following.
\begin{theorem} \label{main-r1}
    Let $\mathbb{T}^3$ be a three-dimensional flat torus. Let us consider $(\varrho,u)$  a weak solution to the Navier-Stokes-Korteweg system (\ref{cont_cap}) - (\ref{mom_cap}) in the sense of Definition \ref{weak-def}.
    Let $(r,U)$ be the  strong solution emanating from the same initial data and such that $(r,U,V,W)$ satisfy \eqref{cont-r} - \eqref{mom-V}. Then, $(\varrho, u) = (r, U)$ in $(0, T) \times \mathbb{T}^3$, which corresponds to a weak-strong uniqueness property.
\end{theorem}

\begin{proof}

We take the strong solutions of the augmented system \eqref{cont-r}-\eqref{mom-V} as test function in the relative entropy \eqref{rei-res} and we multiply (\ref{mom-W}) and (\ref{mom-V}) by $\displaystyle{\frac{\varrho}{r}(W-w)}$ and $\displaystyle{\frac{\varrho}{r}(V-v)}$, respectively. We obtain
$$
\int_0^T\!\! \int_{\mathbb{T}^3} \varrho \partial_t W (W-w) dxdt =
$$
$$
-  \int_0^T\!\! \int_{\mathbb{T}^3} \varrho \nabla W \cdot U (W-w) dxdt
- \int_0^T\!\! \int_{\mathbb{T}^3} \frac{\varrho}{r} \nabla p(r) (W-w) dxdt
$$
$$    + \int_0^T\!\! \int_{\mathbb{T}^3}
    \Big[
    2\nu \frac{\varrho}{r} \textrm{div} (r \nabla W) (W-w)
    - 2\nu \frac{\varrho}{r} \textrm{div} (r \nabla V) (W-w)
   $$
   \begin{equation} \label{mom-W-test}
 + 2\kappa^2
    \varrho \nabla \Delta r (W-w)
    \Big]
    dxdt
\end{equation}
and
$$
\int_0^T \int_{\mathbb{T}^3}
\varrho \partial_t V (V-v)
dxdt
$$
$$    =
    \int_0^T \int_{\mathbb{T}^3}
    \Big[
    - \varrho \nabla V \cdot U (V-v) 
    + 2\nu \frac{\varrho}{r}\textrm{div} (r \nabla^t V) (V-v)
   $$
   \begin{equation} \label{mom-V-test}
 - 2\nu \frac{\varrho}{r} \textrm{div} (r \nabla^t W) (V-v)
    \Big]
    dxdt.
\end{equation}
Now, we multiply (\ref{cont-r}) by $H'(r)$ and we integrate by parts and we obtain
\begin{equation}
\label{p1}
\int_0^T \int_{\mathbb{T}^3}
\left[
H'(r)\partial_t r - r \nabla (H'(r)) \cdot U 
\right]dxdt
= 0.
\end{equation}
By using \eqref{hpressure} we can rewrite \eqref{p1} as  
\begin{equation} \label{obs-0}
\int_0^T \int_{\mathbb{T}^3}
H'(r)\partial_t r dxdt= - \int_0^T \int_{\mathbb{T}^3} p(r) \textrm{div} U dxdt.
\end{equation}
Plugging (\ref{mom-W-test}), (\ref{mom-V-test}) and (\ref{obs-0}) in (\ref{step-3}), we have
$$
\mathcal{E}(T,\cdot) - \mathcal{E}(0,\cdot)
\leq
\int_0^T \int_{\mathbb{T}^3} 
\varrho \nabla W (u-U) \cdot (W-w)
dxdt
$$
$$
+ \int_0^T \int_{\mathbb{T}^3} 
\varrho \nabla V (u-U)\cdot(V-v)
dxdt
$$
$$
+ 2\nu \int_0^T \int_{\mathbb{T}^3}
\frac{\varrho}{r} 
\left(
\textrm{div} (r \nabla W) \cdot (W-w)
+ 
\textrm{div}(r \nabla^t
V) \cdot (V-v)
\right)
dxdt
$$
$$
- 2\nu \int_0^T \int_{\mathbb{T}^3}
\frac{\varrho}{r} \left(
\textrm{div} (r \nabla^t W) (V-v)
+ \textrm{div} (r \nabla V) (W-w)
\right)
dxdt
$$
$$
+ 2\nu \int_0^T \int_{\mathbb{T}^3} \varrho \left( \nabla^t v : \nabla V + \nabla w: \nabla W \right)
dxdt
- 2\nu \int_0^T \int_{\mathbb{T}^3} \varrho
\left(
\nabla v:\nabla W
+  \nabla^t w: \nabla
V
\right)
dxdt
$$
$$
+ 2\kappa^2 \int_0^T \int_{\mathbb{T}^3} \left( \varrho \partial_t \Delta r + \varrho u \cdot \nabla \Delta r - \partial_t r \Delta r + 
\varrho \nabla \Delta r (W-w)\right)
dxdt
$$
$$
- 2\kappa^2 \int_0^T \int_{\mathbb{T}^3} 
\left( 
\nabla\varrho \cdot \nabla( \varrho \textrm{div} W) 
- \frac{1}{2} |\nabla \varrho|^2 \textrm{div} W
+ \nabla \varrho \otimes \nabla \varrho : \nabla W
\right)
dxdt
$$
$$
+ 
\int_0^T \int_{\mathbb{T}^3} \left(  p(r) \textrm{div} U - p(\varrho) \textrm{div} W
\right)
dxdt
$$
$$
-2\nu \int_0^T \int_{\mathbb{T}^3} \left( |\mathcal{S}(u)|^2 + |\mathcal{A}(w)|^2   + \frac{p'(\varrho)}{\varrho} |\nabla \varrho|^2 \right)dxdt 
-4\nu \kappa^2 \int_0^T \int_{\mathbb{T}^3} |\Delta \varrho|^2  dxdt
$$
$$
- \int_0^T \int_{\mathbb{T}^3} \frac{\varrho}{r} \nabla (p(r)) \cdot (W-w)
dxdt
- \int_0^T \int_{\mathbb{T}^3} 
\partial_t (H'(r))(\varrho-r)
dxdt
$$
\begin{equation} \label{step-4}
- \int_0^T \int_{\mathbb{T}^3} \varrho u \nabla(H'(r))
dxdt
+ \int_0^T \int_{\mathbb{T}^3} H'(r) \partial_t r
dxdt.
\end{equation}
Now, we compute
$$
\partial_t (H'(r)) = - p'(r)\textrm{div}U
- H''(r)\nabla(r)\cdot U 
= - p'(r)\textrm{div}U - \nabla(H'(r)) \cdot U.
$$
Consequently, the last four terms in
(\ref{step-4}) can be rewritten as follows
$$
- \int_0^T \int_{\mathbb{T}^3} \frac{\varrho}{r} \nabla (p(r)) \cdot (W-w)
dxdt
- \int_0^T \int_{\mathbb{T}^3} 
\partial_t (H'(r))(\varrho-r)
dxdt
$$
$$
- \int_0^T \int_{\mathbb{T}^3} \varrho u \nabla(H'(r)) dxdt
+ \int_0^T \int_{\mathbb{T}^3} H'(r) \partial_t r
dxdt
$$
\begin{equation} \label{rew}
    = \int_0^T \int_{\mathbb{T}^3} \varrho \nabla (H'(r)) \cdot (v - V)
    dxdt
    + \int_0^T \int_{\mathbb{T}^3} p'(r) \textrm{div}U (\varrho - r)dxdt.
\end{equation}
We end up with
$$
\mathcal{E}(T,\cdot) - \mathcal{E}(0,\cdot)
\leq
\int_0^T \int_{\mathbb{T}^3} 
\varrho \nabla W (u-U) \cdot (W-w)
dxdt
$$
$$
+ \int_0^T \int_{\mathbb{T}^3} 
\varrho \nabla V (u-U)\cdot(V-v)
dxdt
$$
$$
+ 2\nu \int_0^T \int_{\mathbb{T}^3}
\frac{\varrho}{r} 
\left(
\textrm{div} (r \nabla W) \cdot (W-w)
+ 
\textrm{div}(r \nabla^t
V) \cdot (V-v)
\right)
dxdt
$$
$$
- 2\nu \int_0^T \int_{\mathbb{T}^3}
\frac{\varrho}{r} \left(
\textrm{div} (r \nabla^t W) (V-v)
+ \textrm{div} (r \nabla V) (W-w)
\right)
dxdt
$$
$$
+ 2\nu \int_0^T \int_{\mathbb{T}^3} \varrho \left( \nabla^t v : \nabla V + \nabla w: \nabla W  \right)
dxdt
- 2\nu \int_0^T \int_{\mathbb{T}^3} \varrho
\left(
\nabla v:\nabla W
+  \nabla^t w: \nabla
V
\right)
dxdt
$$
$$
+ 2\kappa^2 \int_0^T \int_{\mathbb{T}^3} \left( \varrho \partial_t \Delta r + \varrho u \cdot \nabla \Delta r - \partial_t r \Delta r + 
\varrho \nabla \Delta r (W-w)\right)
dxdt
$$
$$
- 2\kappa^2 \int_0^T \int_{\mathbb{T}^3} 
\left( 
\nabla\varrho \cdot \nabla( \varrho \textrm{div} W) 
- \frac{1}{2} |\nabla \varrho|^2 \textrm{div} W
+ \nabla \varrho \otimes \nabla \varrho : \nabla W
\right)
dxdt
$$
$$
+ 
\int_0^T \int_{\mathbb{T}^3} \left( \left(  
p'(r)(\varrho - r) - p(\varrho) + p(r) \right) \textrm{div}U
\right)
dxdt
$$
$$
-2\nu \int_0^T \int_{\mathbb{T}^3} \left( |\mathcal{S}(u)|^2 + |\mathcal{A}(w)|^2   + \frac{p'(\varrho)}{\varrho} |\nabla \varrho|^2 \right)dxdt 
-4\nu \kappa^2 \int_0^T \int_{\mathbb{T}^3} |\Delta \varrho|^2  dxdt
$$
\begin{equation} \label{step-5}
- \int_0^T \int_{\mathbb{T}^3} \left(\varrho \nabla (H'(r)) \cdot (V - v)
+ p(\varrho) \textrm{div}
V \right) dxdt.
\end{equation}
Now, we handle the viscous terms in \eqref{step-5}. First we notice that
$$
-2\nu \int_0^T \int_{\mathbb{T}^3} 
\left(
|\mathcal{S}(u)|^2
+
|\mathcal{A}(w)|^2 + \frac{p'(\varrho)}{\varrho}|\nabla \varrho|^2
\right)
dxdt
$$
$$
=-2\nu \int_0^T \int_{\mathbb{T}^3}
\varrho
\left(
\left|
\frac{\mathcal{S}(u)}{\sqrt{\varrho}}
-D(U)
\right|^2
+
\left|
\frac{\mathcal{A}(w)}{\sqrt{\varrho}}
-A(W)
\right|^2
\right)
dxdt
$$
$$
+2\nu \int_0^T \int_{\mathbb{T}^3}
\left(
\varrho|A(W)|^2
-\sqrt{\varrho}\mathcal{A}(w)A(W)
+\varrho|D(U)|^2
-\sqrt{\varrho}\mathcal{S}(u)D(U)
\right)
dxdt
$$
$$
-2\nu \int_0^T \int_{\mathbb{T}^3}
\left(
\sqrt{\varrho}\mathcal{A}(w)A(W)
+\sqrt{\varrho}\mathcal{S}(u)D(U)
\right)
dxdt
$$
$$
-2\nu \int_0^T \int_{\mathbb{T}^3}
\varrho(p'(\varrho)\nabla\log\varrho-p'(r)\nabla\log r)\cdot
(\nabla \log \varrho - \nabla \log r)
dxdt
$$
\begin{equation} \label{visc-1}
+2\nu \int_0^T \int_{\mathbb{T}^3}
\frac{\varrho}{r}p'(r)
\nabla r
\left(\frac{\nabla r}{r}-\frac{\nabla \varrho}{\varrho}\right)
-2\nu \int_0^T \int_{\mathbb{T}^3}
p'(\varrho)\nabla \varrho \cdot \frac{\nabla r}{r}
dxdt.
\end{equation}
The following viscous terms 
$$
2\nu \int_0^T \int_{\mathbb{T}^3} \varrho \left( \nabla^t v : \nabla V + \nabla w: \nabla W  \right)dxdt
- 2\nu \int_0^T \int_{\mathbb{T}^3} \varrho
\left(
\nabla v:\nabla W
+  \nabla^t w: \nabla
V
\right)
dxdt
$$
should be interpreted as:
$$
2\nu\int_0^T \int_{\Omega} \varrho \nabla^t v : \nabla V dxdt+2\nu \int_0^T \int_{\Omega}
\varrho \nabla w: \nabla W dxdt
$$
$$
= 2\nu\int_0^T \int_{\Omega}
\left(
\sqrt{\varrho}\mathcal{S}(v):D(V)
+\sqrt{\varrho}\mathcal{S}(v):A(V)
-\sqrt{\varrho}\mathcal{A}(v):D(V)
-\sqrt{\varrho}\mathcal{A}(v)A(V)
\right)
dxdt;
$$
$$
+ 2\nu\int_0^T \int_{\Omega}
\left(
\sqrt{\varrho}\mathcal{S}(w):D(W)
+\sqrt{\varrho}\mathcal{S}(w):A(W)
+\sqrt{\varrho}\mathcal{A}(w):D(W)
+\sqrt{\varrho}\mathcal{A}(w)A(W)
\right)
dxdt;
$$
and

$$
2\nu \int_0^T \int_{\Omega}
\varrho \nabla v: \nabla W dxdt+2\nu\int_0^T \int_{\Omega} \varrho \nabla^t w : \nabla V dxdt
$$
$$
+2\nu\int_0^T \int_{\Omega}
\left(
\sqrt{\varrho}\mathcal{S}(v):D(W)
+\sqrt{\varrho}\mathcal{S}(v):A(W)
+\sqrt{\varrho}\mathcal{A}(v):D(W)
+\sqrt{\varrho}\mathcal{A}(v)A(W)
\right)
dxdt;
$$
$$
= 2\nu\int_0^T \int_{\Omega}
\left(
\sqrt{\varrho}\mathcal{S}(w):D(V)
+\sqrt{\varrho}\mathcal{S}(w):A(V)
-\sqrt{\varrho}\mathcal{A}(w):D(V)
-\sqrt{\varrho}\mathcal{A}(w)A(V)
\right)
dxdt.
$$
We conclude analyzing the remaining viscous terms
$$
I_{1}=2\nu \int_0^T \int_{\mathbb{T}^3}
\frac{\varrho}{r} 
\left(
\textrm{div} (r \nabla W) \cdot (W-w)
+ 
\textrm{div}(r \nabla^t
V) \cdot (V-v)
\right)
dxdt
$$
$$
I_{2}=- 2\nu \int_0^T \int_{\mathbb{T}^3}
\frac{\varrho}{r} \left(
\textrm{div} (r \nabla^t W) \cdot (V-v)
+ \textrm{div} \cdot (r \nabla V) (W-w)
\right)
dxdt.
$$
By parts integration gives
$$
I_{1}- 2\nu \int_0^T \int_{\mathbb{T}^3}
\frac{\varrho}{r} 
\left(
(r \nabla W) \cdot \nabla (W-w)
+ 
(r \nabla^t V) \cdot \nabla (V-v)
\right)
dxdt
$$
\begin{equation} \label{visc-11}
- 2\nu \int_0^T \int_{\mathbb{T}^3}
\nabla \left(
\frac{\varrho}{r} \right)
\left(
(r \nabla W) \cdot  (W-w)
+ 
(r \nabla^t
V) \cdot  (V-v)
\right)
dxdt
\end{equation}
and
$$
I_{2}=2\nu \int_0^T \int_{\mathbb{T}^3}
\frac{\varrho}{r} \left(
 (r \nabla^t W) \cdot \nabla (V-v)
+ (r \nabla V) \cdot \nabla (W-w)
\right)
dxdt
$$
\begin{equation} \label{visc-22}
+ 2\nu \int_0^T \int_{\mathbb{T}^3}
\nabla \left(
\frac{\varrho}{r} \right)
\left(
 (r \nabla^t W) \cdot (V-v)
+ (r \nabla V) \cdot (W-w)
\right)
dxdt.
\end{equation}
Taking into account that $A(V)=0$ in  the terms (\ref{visc-11}) and (\ref{visc-22}) we end up with,
$$
I_{1}+I_{2}=-2\nu \int_0^T \int_{\mathbb{T}^3}
\left[
\varrho|D(W)|^2
-\sqrt{\varrho}\mathcal{S}(w)D(W)
+\varrho|A(W)|^2
-\sqrt{\varrho}\mathcal{A}(w)A(W)
\right]
dxdt
$$
$$
-2\nu \int_0^T \int_{\mathbb{T}^3}
\left[
\varrho|D(V)|^2
-\sqrt{\varrho}\mathcal{S}(v)D(V)
\right]
dxdt
$$
\begin{equation} \label{visc-11-rew1}
- 2\nu \int_0^T \int_{\mathbb{T}^3}
\nabla \left(
\frac{\varrho}{r} \right)
\left(
(r (D(W)+A(W))) \cdot  (W-w)
+ 
(r \nabla^t
V) \cdot  (V-v)
\right)
dxdt,
\end{equation}
$$
2\nu \int_0^T \int_{\mathbb{T}^3}
\left[
\varrho D(W)D(V)
-\sqrt{\varrho}\mathcal{S}(v)D(W)
+\sqrt{\varrho}\mathcal{A}(v)A(W)
\right]
dxdt
$$
$$
+ 2\nu \int_0^T \int_{\mathbb{T}^3}
\left[
\varrho D(W)D(V)
-\sqrt{\varrho}\mathcal{S}(w)D(V)
\right]
dxdt
$$
\begin{equation} \label{visc-22-rew2}
+ 2\nu \int_0^T \int_{\mathbb{T}^3}
\nabla \left(
\frac{\varrho}{r} \right)
\left(
(r \nabla^t W) \cdot (V-v)
+ (r \nabla V) \cdot (W-w)
\right)
dxdt.
\end{equation}

\noindent
Now, noting that
$$
2\nu \int_0^T \int_{\mathbb{T}^3} \frac{\varrho}{r} p'(r) \nabla r \cdot \left( \frac{\nabla r}{r} - \frac{\nabla \varrho}{\varrho} \right) dxdt
- 2\nu \int_0^T \int_{\mathbb{T}^3} p'(\varrho) \nabla \varrho \cdot \frac{\nabla r}{r} dxdt
$$
\begin{equation*} 
- \int_0^T \int_{\mathbb{T}^3} 
\left(
\varrho \nabla (H'(r)) \cdot (V - v)
+ p(\varrho) \textrm{div}
V
\right)
dxdt=0,
\end{equation*}
and observing again that $A(V)=0$ and that the scalar product between symmetric and skew-symmetric second-order tensors is null, we rewrite (\ref{step-5}) as follows
$$
\mathcal{E}(T,\cdot) - \mathcal{E}(0,\cdot)
$$
$$
+2\nu \int_0^T \int_{\mathbb{T}^3}
\varrho
\left(
\left|
\frac{\mathcal{S}(u)}{\sqrt{\varrho}}
-D(U)
\right|^2
+
\left|
\frac{\mathcal{A}(w)}{\sqrt{\varrho}}
-A(W)
\right|^2
\right)
dxdt
+4\nu \kappa^2 \int_0^T \int_{\mathbb{T}^3} |\Delta \varrho|^2  dxdt
$$
$$
+2\nu \int_0^T \int_{\mathbb{T}^3}
\varrho(p'(\varrho)\nabla\log\varrho-p'(r)\nabla\log r)\cdot
(\nabla \log \varrho - \nabla \log r)
dxdt
$$
$$
\leq
\int_0^T \int_{\mathbb{T}^3} \varrho \nabla W (u-U) \cdot (W-w)dxdt+ \int_0^T \int_{\mathbb{T}^3} \varrho \nabla V (u-U)\cdot(V-v)dxdt
$$
$$
- 2\nu \int_0^T \int_{\mathbb{T}^3}\nabla \left(\frac{\varrho}{r} \right)\left((r (D(W)+A(W))) \cdot  (W-w)+ (r \nabla^t V) \cdot  (V-v)\right)dxdt
$$
$$
+ 2\nu \int_0^T \int_{\mathbb{T}^3}\nabla \left(\frac{\varrho}{r} \right)\left(r \nabla^t W) \cdot (V-v)+ (r \nabla V) \cdot (W-w)\right)dxdt
$$
$$
+ 2\kappa^2 \int_0^T \int_{\mathbb{T}^3} \left( \varrho \partial_t \Delta r + \varrho u \cdot \nabla \Delta r - \partial_t r \Delta r + 
\varrho \nabla \Delta r (W-w)\right)
dxdt
$$
$$
- 2\kappa^2 \int_0^T \int_{\mathbb{T}^3} 
\left( 
\nabla\varrho \cdot \nabla( \varrho \textrm{div} W) 
- \frac{1}{2} |\nabla \varrho|^2 \textrm{div} W
+ \nabla \varrho \otimes \nabla \varrho : \nabla W
\right)
dxdt
$$
\begin{equation} \label{step-6}
+ 
\int_0^T \int_{\mathbb{T}^3} \left( \left(  
p'(r)(\varrho - r) - p(\varrho) + p(r) \right) \textrm{div}U
\right)
dxdt.
\end{equation}

\noindent
Now we consider the capillary terms.  First, we have (formally)
$$
- 2\kappa^2 \int_0^T \int_{\mathbb{T}^3}\left( 
\nabla\varrho \cdot \nabla( \varrho \textrm{div} W) - \frac{1}{2} |\nabla \varrho|^2 \textrm{div} W+ \nabla \varrho \otimes \nabla \varrho : \nabla W\right)
dxdt
$$
$$
= -2\kappa^2
\int_0^T \int_{\mathbb{T}^3}
\varrho W \nabla \Delta \varrho
dxdt.
$$
Consequently, 
$$
2\kappa^2 \int_0^T \int_{\mathbb{T}^3} \left( \varrho \partial_t \Delta r + \varrho u \cdot \nabla \Delta r - \partial_t r \Delta r + 
\varrho \nabla \Delta r (W-w) \right)
dxdt
$$
$$
- 2\kappa^2 \int_0^T \int_{\mathbb{T}^3}\left( \nabla\varrho \cdot \nabla( \varrho \textrm{div} W) - \frac{1}{2} |\nabla \varrho|^2 \textrm{div} W
+ \nabla \varrho \otimes \nabla \varrho : \nabla W\right)dxdt
$$
$$
= 2\kappa^2 \int_0^T \int_{\mathbb{T}^3} 
\left(
\varrho \partial_t \Delta r
- \partial_t r \Delta r
+\varrho (u-w) \cdot \nabla \Delta r
+\varrho W \nabla \Delta (r-\varrho)
\right)
dxdt.
$$
$$
=2\kappa^2 \int_0^T \int_{\mathbb{T}^3} \Big(\varrho \partial_t \Delta r
- \partial_t r \Delta r
+\varrho (u-w) \cdot \nabla \Delta r
$$
\begin{equation}
\label{p2}
+r W \nabla \Delta (r-\varrho)
+ W(\varrho - r) \nabla \Delta (r-\varrho)\Big)
dxdt.
\end{equation}
After by parts integration and by recalling that
$\partial_t r + \text{div}(rW)=2\nu\Delta r$, we have \eqref{p2} is equal to
\begin{equation}
2\kappa^2 \int_0^T \int_{\mathbb{T}^3} 
\left(
\varrho (u-w) \cdot \nabla \Delta r
- 2\nu \Delta r \Delta (r-\varrho)
+ W(\varrho - r) \nabla \Delta (r-\varrho)
\right)
dxdt.
\label{p3}
\end{equation}
From the definition of $v = 2\nu \nabla \log \varrho$, we rewrite \eqref{p3} as follows
$$
 - 2\kappa^2 \int_0^T \int_{\mathbb{T}^3} \varrho v \cdot \nabla \Delta r
dxdt
- 4\kappa^2\nu \int_0^T \int_{\mathbb{T}^3}  \Delta r \Delta (r-\varrho)
dxdt
+ 2\kappa^2\int_0^T \int_{\mathbb{T}^3} W(\varrho - r) \nabla \Delta (r-\varrho)
dxdt
$$
$$
= - 4\kappa^2\nu \int_0^T \int_{\mathbb{T}^3} \nabla \varrho \cdot \nabla \Delta r
dxdt
- 4\kappa^2\nu \int_0^T \int_{\mathbb{T}^3}  \Delta r \Delta (r-\varrho)
dxdt
+ 2\kappa^2 \int_0^T \int_{\mathbb{T}^3} W(\varrho - r) \nabla \Delta (r-\varrho)
dxdt.
$$
$$
=4\kappa^2\nu \int_0^T \int_{\mathbb{T}^3} \Delta \varrho \cdot \Delta r
dxdt
- 4\kappa^2\nu \int_0^T \int_{\mathbb{T}^3}  \Delta r \Delta (r-\varrho)
dxdt
+ 2\kappa^2 \int_0^T \int_{\mathbb{T}^3} W(\varrho - r) \nabla \Delta (r-\varrho)
dxdt
$$
\begin{equation}
=  8\kappa^2\nu \int_0^T \int_{\mathbb{T}^3} \Delta \varrho \cdot \Delta r
dxdt
- 4\kappa^2\nu \int_0^T \int_{\mathbb{T}^3}  |\Delta r|^2
dxdt
+ 2\kappa^2 \int_0^T \int_{\mathbb{T}^3} W(\varrho - r) \nabla \Delta (r-\varrho)
dxdt.
\label{p4}
\end{equation}
Finally, the last term in \eqref{p4} can be handled as
$$
2\kappa^2 \int_0^T \int_{\mathbb{T}^3} W(\varrho - r) \nabla \Delta (r-\varrho)
dxdt
$$
$$
= 2\kappa^2 \int_0^T \int_{\mathbb{T}^3}
W \text{div}
\left[
(\varrho-r)\Delta(\varrho-r)\mathbb{I}-\frac{1}{2}|\nabla(\varrho-r)|^2\mathbb{I}
+\nabla(\varrho-r)\otimes\nabla(\varrho-r)
\right]
dxdt
$$
$$
= -2\kappa^2 \int_0^T \int_{\mathbb{T}^3}
\nabla(\varrho-r)\otimes\nabla(\varrho-r):\nabla W
dxdt
+ 3\kappa^2
\int_0^T \int_{\mathbb{T}^3}
\text{div}W|\nabla(\varrho-r)|^2
dxdt
$$
$$
+2\kappa^2
\int_0^T \int_{\mathbb{T}^3}
(\varrho-r)
\nabla\text{div}W
\nabla(\varrho-r)
dxdt.
$$
Consequently, we can rewrite the relation (\ref{step-6}) as follows
$$
\mathcal{E}(T,\cdot) - \mathcal{E}(0,\cdot)
$$
$$
+2\nu \int_0^T \int_{\mathbb{T}^3}\varrho\left(\left|\frac{\mathcal{S}(u)}{\sqrt{\varrho}}-D(U)\right|^2+\left|\frac{\mathcal{A}(w)}{\sqrt{\varrho}}
-A(W)\right|^2\right)dxdt+4\nu \kappa^2 \int_0^T \int_{\mathbb{T}^3} |\Delta \varrho-\Delta r|^2  dxdt
$$
$$
+2\nu \int_0^T \int_{\mathbb{T}^3}\varrho(p'(\varrho)\nabla\log\varrho-p'(r)\nabla\log r)\cdot(\nabla \log \varrho - \nabla \log r)dxdt
$$
$$
\leq\int_0^T \int_{\mathbb{T}^3} \varrho \nabla W (u-U) \cdot (W-w)dxdt+ \int_0^T \int_{\mathbb{T}^3} 
\varrho \nabla V (u-U)\cdot(V-v)dxdt
$$
$$
- 2\nu \int_0^T \int_{\mathbb{T}^3}\nabla \left(\frac{\varrho}{r} \right)\left((r (D(W)+A(W))) \cdot  (W-w)+ (r \nabla^t V) \cdot  (V-v)\right)dxdt
$$
$$
+ 2\nu \int_0^T \int_{\mathbb{T}^3}\nabla \left(\frac{\varrho}{r} \right)\left((r \nabla^t W) \cdot (V-v)+ (r \nabla V) \cdot (W-w)\right)dxdt
$$
$$
-2\kappa^2 \int_0^T \int_{\mathbb{T}^3}\nabla(\varrho-r)\otimes\nabla(\varrho-r):\nabla Wdxdt+ 3\kappa^2\int_0^T \int_{\mathbb{T}^3}\text{div}W|\nabla(\varrho-r)|^2dxdt
$$
$$
+2\kappa^2\int_0^T \int_{\mathbb{T}^3}(\varrho-r)\nabla\text{div}W\nabla(\varrho-r)dxdt
$$
\begin{equation} \label{step-7}
+ \int_0^T \int_{\mathbb{T}^3} \left( \left(  p'(r)(\varrho - r) - p(\varrho) + p(r) \right) \textrm{div}U\right)dxdt.
\end{equation}
The last step is to work on the remainders term on the right hand side of \eqref{step-7}
\subsubsection*{Convective and pressure terms}
\noindent
We have
$$
\int_0^T \int_{\mathbb{T}^3} 
\varrho \nabla W (u-U) \cdot (W-w)
dxdt
+ \int_0^T \int_{\mathbb{T}^3} 
\varrho \nabla V (u-U)\cdot(V-v)
dxdt
$$
$$
+ 
\int_0^T \int_{\mathbb{T}^3} \left( \left(  
p'(r)(\varrho - r) - p(\varrho) + p(r) \right) \textrm{div}U
\right)
dxdt
$$
$$
\leq
C\int_0^T \mathcal{E}(\cdot,t)dt.
$$
\subsubsection*{Viscous terms}

\noindent
Observing that
$$
2\nu
\nabla \left(\frac{\varrho}{r}\right) =
2\nu
\left[
\frac{\varrho}{r}\left(\nabla \log \varrho - \nabla \log r \right)
\right]
= \frac{\varrho}{r}
\left(
v-V
\right),
$$
we have
$$
- 2\nu \int_0^T \int_{\mathbb{T}^3}
\nabla \left(
\frac{\varrho}{r} \right)
\left(
(r (D(W)+A(W))) \cdot  (W-w)
+ 
(r \nabla^t
V) \cdot  (V-v)
\right)
dxdt
$$
$$
+ 2\nu \int_0^T \int_{\mathbb{T}^3}
\nabla \left(
\frac{\varrho}{r} \right)
\left(
(r \nabla^t W) \cdot (V-v)
+ (r \nabla V) \cdot (W-w)
\right)
dxdt
$$
$$
\leq
C\int_0^T \mathcal{E}(\cdot,t)dt.
$$
\subsubsection*{Capillary terms}

\noindent
We have
$$
-2\kappa^2 \int_0^T \int_{\mathbb{T}^3}
\nabla(\varrho-r)\otimes\nabla(\varrho-r):\nabla W
dxdt
+ 3\kappa^2
\int_0^T \int_{\mathbb{T}^3}
\text{div}W|\nabla(\varrho-r)|^2
dxdt
$$
$$
+2\kappa^2
\int_0^T \int_{\mathbb{T}^3}
(\varrho-r)
\nabla\text{div}W
\nabla(\varrho-r)
dxdt
$$
$$
\leq
C(\kappa)\int_0^T \mathcal{E}(\cdot,t)dt.
$$
\subsubsection*{Relative entropy pressure term}

\noindent
We need to handle the term
$$
2\nu \int_0^T \int_{\mathbb{T}^3}
\varrho(p'(\varrho)\nabla\log\varrho-p'(r)\nabla\log r)\cdot
(\nabla \log \varrho - \nabla \log r)
dxdt
$$
which, a priori, has no sign.
We compute
\begin{multline} \label{press-contr}
\varrho (p'(\varrho)\nabla \log \varrho - p'(r)\nabla \log r)(\nabla \log \varrho - \nabla \log r) \\
= \varrho p'(\varrho) |\nabla \log \varrho - \nabla \log r|^2 + \varrho (p'(\varrho) - p'(r))\nabla \log r (\nabla \log \varrho - \nabla \log r)\\
=\varrho p'(\varrho) |\nabla \log \varrho - \nabla \log r|^2 + \nabla \left[p(\varrho) - p(r) - p'(r)(\varrho - r)\right] \nabla \log r \\ - \left[\varrho(p'(\varrho) - p'(r)) - p''(r)(\varrho - r)r\right]|\nabla\log r|^2.
\end{multline} 
The first term on the right hand side of (\ref{press-contr}) is positive and thus can be neglected. For the remaining terms, integrating by parts, we have
\begin{multline}
\int_0^T \int_{\Omega}
\nabla \left[p(\varrho) - p(r) - p'(r)(\varrho - r)\right] \nabla \log r\\ - \left[\varrho(p'(\varrho) - p'(r)) - p''(r)(\varrho - r)r 
\right] |\nabla\log r|^2 dxdt \\
= - \int_0^T \int_{\Omega} |(p(\varrho) - p(r) - p'(r) (\varrho - r)| |\Delta \log r| dxdt\\
- \int_0^T \int_{\Omega} \left[\varrho (p'(\varrho) - p'(r)) - p''(r)(\varrho - r) r \right] |\nabla \log r|^2 dxdt.
\end{multline}
Now, observing that
$$
\left[\varrho (p'(\varrho) - p'(r)) - p''(r)(\varrho - r)r\right]
\approx
H(\varrho|r)
$$
we end up with 
$$
2\nu \int_0^T \int_{\mathbb{T}^3}
\varrho(p'(\varrho)\nabla\log\varrho-p'(r)\nabla\log r)\cdot
(\nabla \log \varrho - \nabla \log r)
dxdt
$$
$$
\leq C(\nu) \int_0^T \mathcal{E}(t,\cdot) dt.
$$
Finally, by putting all the previous steps together we end up with
$$\mathcal{E}(T,\cdot) - \mathcal{E}(0,\cdot)
\leq C(\nu) \int_0^T \mathcal{E}(t,\cdot) dt.$$
We conclude the proof  by  applying Gronwall's Lemma.
\end{proof}

\subsection{Weak-strong uniqueness for pressureless capillary fluids}
\label{WSpl}
\noindent
As in Section \ref{ws-un} we introduce the augmented pressureless system satisfied by the strong solutions,
\begin{equation} \label{cont-r-pless}
    \partial_t r + \textrm{div}(rU) =0,
\end{equation}
\begin{equation} \label{mom-W-pless}
    r \left( \partial_t W + \nabla W \cdot U \right) - 2\nu \textrm{div} (r \nabla W)
    + 2\nu \textrm{div} (r \nabla V)
    - 2\kappa^2
    r \nabla \Delta r =0,
\end{equation}
\begin{equation} \label{mom-V-pless}
    r \left( \partial_t V + \nabla V \cdot U \right)
    - 2\nu \textrm{div} (r \nabla^t V)
    + 2\nu \textrm{div} (r \nabla^t W) = 0
\end{equation}
with $U = W-V $ and $V=2\nu\nabla\log r$ and $(r,W,V)$ satisfying the same regularity as in (\ref{reg_class}).
Our main result of this section is the following theorem.

\begin{theorem} \label{main-r2}
    Let $\mathbb{T}^3$ be a three-dimensional flat torus. Let us consider $(\varrho,u)$  a weak solution to the pressureless Navier-Stokes-Korteweg system (\ref{cont_cap-press}) - (\ref{mom_cap-press})  in the sense of Definition \ref{weak-def}.
Let $(r,U)$ be the  strong solution emanating from the same initial data and such that $(r,U,V,W)$ satisfy \eqref{cont-r-pless}-\eqref{mom-V-pless} and the regularity class (\ref{reg_class}). Then, $(\varrho, u) = (r, U)$ in $(0, T) \times \mathbb{T}^3$, which corresponds to a weak-strong uniqueness property.
\end{theorem}

As for the capillary flow, it holds the property that any strong solution of (\ref{cont_cap-press}) - (\ref{mom_cap-press}) is a strong solution of the augmented system \eqref{cont-r-pless} - \eqref{mom-V-pless}. According to our knowledge there are no results for global solutions of (\ref{cont_cap-press}) - (\ref{mom_cap-press}) and large general initial data. Hence the solutions of \eqref{cont-r-pless} - \eqref{mom-V-pless} have to be understood local in time and with very regular initial data.
However it is possible to prove the existence of a global strong solution for the capillary pressureless system in the case of  an irrotational initial velocity. Indeed in \cite{CaDo} the authors proved a link between the pressureless system (\ref{cont_cap-press}) - (\ref{mom_cap-press}) and the heat equation for a positive initial density, $\varrho_0 > 0$, $\varrho_0 \in L^1(\mathbb{T}^3)$, and the initial velocity expressed as a gradient of a given potential. It turns out that for those given initial data, the solution of the pressureless system is such that $(\varrho, u) \in C^\infty(0,T; \mathbb{T}^3)$ with $\varrho$ solving the  heat equation. For completeness we recall the result of \cite{CaDo} (see Theorem 2.5 in \cite{CaDo}).

\begin{theorem}
 \label{heat_eq}
	Let $\Omega=\mathbb{T}^{d}$ ($d=2$ or 3) be a periodic domain. Let $\varrho_{0}\in L^{1}\left(\Omega\right)$ with $\varrho_{0}>0$ and continuous. Assume also that $u_{0}=-\nabla\widetilde{\phi}\left(\varrho_{0}\right)$, 
	$\nabla \widetilde{\phi}(\varrho_{0})= \nabla \mu(\varrho_{0})/\varrho_{0}$. Then,  there exists a global weak solution $\left(\varrho,u=-\nabla\widetilde{\phi}\left(\varrho\right)\right)$ of the system (\ref{cont_cap-press}) - (\ref{mom_cap-press}), with $\left(\varrho,u\right)\in C^{\infty}\left(0,T;\Omega\right)$ and $\varrho$ solving the following heat equation almost everywhere	
\begin{equation} \label{heat}
	\partial_{t}\varrho-2\nu\Delta\varrho=0,
	\end{equation}
$$\varrho\left(0,\cdot\right)=\varrho_{0}.$$
\end{theorem}

\begin{remark}
Let us point out  that any solution  of \eqref{heat} with $\varrho_{0}>0$ and  such that  $\varrho \in C^{2}\left(0,T;\Omega\right)$ is a classical solution of (\ref{cont_cap-press}) - (\ref{mom_cap-press}).
\end{remark}

Let us mention that such kind of classical solutions for pressureless systems with an irrotational large initial velocity  (hence an irrotational velocity at any time) are interesting since one can work around this particular solutions in order to obtain global strong solution with large initial data for compressible fluids with density dependent initial data, see \cite{BH-1}, \cite{BH-2}.

Motivated by the Theorem \ref{heat_eq}, the next result is a direct consequence of Theorem \ref{main-r2}. 
\begin{corollary} \label{cor}
    Let $\mathbb{T}^3$ be a three-dimensional flat torus. Let us consider $(\varrho,u)$ be a weak solutions to the pressureless Navier-Stokes-Korteweg system (\ref{cont_cap-press}) - (\ref{mom_cap-press}).
    Let $(r,U)$ be the  strong solution  emanating from the initial data $(r_{0}, U_{0})$ such that
    \begin{equation} \label{id-irr}
        (r_0, U_0) = (r_0, -\nabla \widetilde{\phi}(r_0)), \quad  r_0 > 0, \quad r_0 \in L^1(\mathbb{T}^3)\cap C(\mathbb{T}^3),
    \end{equation}
    with
    \begin{equation*}
        \nabla \widetilde{\phi}(r_0)) = \frac{\nabla \mu(r_0)}{r_0}.
    \end{equation*}
If $(\varrho_{0}, u_{0})=(r_{0},u_{0})$, then $(\varrho, u) = (r, U)$ in $(0, T) \times \mathbb{T}^3$, which corresponds to a weak-strong uniqueness property.
\end{corollary}

We highlight that the result described in the Corollary \ref{cor} is new and interesting in the framework of weak strong uniqueness results. Indeed we proved a weak strong uniqueness property starting from an initial datum for the density  which is merely $L^{1}$ in contrast with the higher regularity that is usually required. Moreover, from the Corollary \ref{cor} we deduce that in the pressureless case in our weak/strong solution vacuum states may not appear for $t>0$ in the non vacuum regions since the initial density is strictly positive and $\varrho=r$ solves the heat equation.

We conclude this section with the proof of the Theorem \ref{main-r2}.

\begin{proof} {\em of the Theorem \ref{main-r2}.}
We multiply (\ref{mom-W-pless}) and (\ref{mom-V-pless}) by $\displaystyle{\frac{\varrho}{r}(W-w)}$ and $\displaystyle{\frac{\varrho}{r}(V-v)}$, respectively. We obtain
$$
\int_0^T \int_{\mathbb{T}^3} \varrho \partial_t W (W-w) dxdt =
$$
$$
-  \int_0^T \int_{\mathbb{T}^3} \varrho \nabla W \cdot U (W-w) dxdt
$$
$$    + \int_0^T \int_{\mathbb{T}^3}
    \Big[
    2\nu \frac{\varrho}{r} \textrm{div} (r \nabla W) (W-w)
    - 2\nu \frac{\varrho}{r} \textrm{div} (r \nabla V) (W-w)$$
  \begin{equation} \label{mom-W-test-pless}
  + 2\kappa^2
    \varrho \nabla \Delta r (W-w)
    \Big]
    dxdt
\end{equation}
and
$$
\int_0^T \int_{\mathbb{T}^3}
\varrho \partial_t V (V-v)
dxdt
$$
$$=
    \int_0^T \int_{\mathbb{T}^3}
    \Big[
    - \varrho \nabla V \cdot U (V-v) 
    + 2\nu \frac{\varrho}{r}\textrm{div} (r \nabla^t V) (V-v)
  $$
  \begin{equation} \label{mom-V-test-pless}
    - 2\nu \frac{\varrho}{r} \textrm{div} (r \nabla^t W) (V-v)
    \Big]
    dxdt.
\end{equation}
Plugging (\ref{mom-W-test-pless}) and (\ref{mom-V-test-pless}) in (\ref{step-3-pless}), we have
$$
\mathcal{E}(T,\cdot) - \mathcal{E}(0,\cdot)
\leq
\int_0^T \int_{\mathbb{T}^3} 
\varrho \nabla W (u-U) \cdot (W-w)
dxdt
$$
$$
+ \int_0^T \int_{\mathbb{T}^3} 
\varrho \nabla V (u-U)\cdot(V-v)
dxdt
$$
$$
+ 2\nu \int_0^T \int_{\mathbb{T}^3}
\frac{\varrho}{r} 
\left(
\textrm{div} (r \nabla W) \cdot (W-w)
+ 
\textrm{div}(r \nabla^t
V) \cdot (V-v)
\right)
dxdt
$$
$$
- 2\nu \int_0^T \int_{\mathbb{T}^3}
\frac{\varrho}{r} \left(
\textrm{div} (r \nabla^t W) (V-v)
+ \textrm{div} (r \nabla V) (W-w)
\right)
dxdt
$$
$$
+ 2\nu \int_0^T \int_{\mathbb{T}^3} \varrho \left( \nabla^t v : \nabla V + \nabla w: \nabla W \right)
dxdt
- 2\nu \int_0^T \int_{\mathbb{T}^3} \varrho\left(\nabla v:\nabla W+  \nabla^t w: \nabla V\right)dxdt
$$
$$
+ 2\kappa^2 \int_0^T \int_{\mathbb{T}^3} \left( \varrho \partial_t \Delta r + \varrho u \cdot \nabla \Delta r - \partial_t r \Delta r + 
\varrho \nabla \Delta r (W-w)\right)
dxdt
$$
$$
- 2\kappa^2 \int_0^T \int_{\mathbb{T}^3} 
\left( 
\nabla\varrho \cdot \nabla( \varrho \textrm{div} W) 
- \frac{1}{2} |\nabla \varrho|^2 \textrm{div} W
+ \nabla \varrho \otimes \nabla \varrho : \nabla W
\right)
dxdt
$$
\begin{equation} \label{step-4-pless}
-2\nu \int_0^T \int_{\mathbb{T}^3} \left( |\mathcal{S}(u)|^2 + |\mathcal{A}(w)|^2   \right)dxdt 
-4\nu \kappa^2 \int_0^T \int_{\mathbb{T}^3} |\Delta \varrho|^2  dxdt
\end{equation}
%
The rest of the computation follows the same line of arguments as for the system with the presence of the pressure term. In particular, the viscous terms can be handled in the same way as in (\ref{visc-11}) - (\ref{visc-22-rew2}). Consequently, we rewrite (\ref{step-4-pless}) as follows
$$
\mathcal{E}(T,\cdot) - \mathcal{E}(0,\cdot)
$$
$$
+2\nu \int_0^T \int_{\mathbb{T}^3}\varrho\left(\left|\frac{\mathcal{S}(u)}{\sqrt{\varrho}}-D(U)\right|^2+\left|\frac{\mathcal{A}(w)}{\sqrt{\varrho}}
-A(W)\right|^2\right)dxdt+4\nu \kappa^2 \int_0^T \int_{\mathbb{T}^3} |\Delta \varrho|^2  dxdt
$$
$$
\leq \int_0^T \int_{\mathbb{T}^3} \varrho \nabla W (u-U) \cdot (W-w)dxdt
+ \int_0^T \int_{\mathbb{T}^3} \varrho \nabla V (u-U)\cdot(V-v)dxdt
$$
$$
- 2\nu \int_0^T \int_{\mathbb{T}^3}\nabla \left(\frac{\varrho}{r} \right)\left((r (D(W)+A(W))) \cdot  (W-w)+ (r \nabla^tV) \cdot  (V-v)\right)dxdt
$$
$$
+ 2\nu \int_0^T \int_{\mathbb{T}^3}\nabla \left(\frac{\varrho}{r} \right)\left((r \nabla^t W) \cdot (V-v)+ (r \nabla V) \cdot (W-w)\right)dxdt
$$
$$
+ 2\kappa^2 \int_0^T \int_{\mathbb{T}^3} \left( \varrho \partial_t \Delta r + \varrho u \cdot \nabla \Delta r - \partial_t r \Delta r + 
\varrho \nabla \Delta r (W-w)\right)dxdt
$$
\begin{equation} \label{step-6-pless}
- 2\kappa^2 \int_0^T \int_{\mathbb{T}^3} \left( \nabla\varrho \cdot \nabla( \varrho \textrm{div} W) - \frac{1}{2} |\nabla \varrho|^2 \textrm{div} W
+ \nabla \varrho \otimes \nabla \varrho : \nabla W\right)dxdt.
\end{equation}
We conclude the proof of Theorem \ref{main-r2} 
following the same arguments developed in Section \ref{ws-un}.
\end{proof}

\section{Application 2: High-Mach number limit}
\label{HM}
\noindent
This section is devoted to the proof of the high Mach number limit for the scaled system \eqref{cont_cap-rs}-\eqref{mom_cap-rs}. The high Mach number regime for a fluid corresponds to the physical state in which the fluid speed increases beyond the sound speed, then the compressibility effect become relevant (highly compressible fluids). This behaviour can be  formally seen by sending $\varepsilon\to 0$ (or equivalently $\mathcal{M}a\to \infty)$) in  \eqref{cont_cap-rs}-\eqref{mom_cap-rs} and ending up with  the pressureless system \eqref{cont_cap-press}-\eqref{mom_cap-press}. In \cite{CaDo} the authors prove rigorously the convergence of weak solutions of the scaled Navier-Stokes-Korteweg system towards the weak solutions of the pressureless system.

Here, we complete the theory by showing a weak-strong convergence results. Indeed  by using the relative entropy functional we will prove the convergence of the weak solutions of the capillary fluid system \eqref{cont_cap-rs}-\eqref{mom_cap-rs} towards the strong solutions of the pressureless fluid capillary fluids \eqref{cont_cap-press}-\eqref{mom_cap-press}. We notice that, because of the Theorem \ref{heat_eq} this class of solutions is not empty.

For the scaled system \eqref{cont_cap-rs}-\eqref{mom_cap-rs} the associated augmented system is given by
\begin{equation} \label{ags_1.0}
    \partial_t \varrho_{\varepsilon}  + \textrm{div} (\varrho_{\varepsilon} u_{\varepsilon}) = 0,
\end{equation}
$$
\partial_t (\varrho_{\varepsilon} w_{\varepsilon}) + \textrm{div}(\varrho_{\varepsilon} w_{\varepsilon} \otimes u_{\varepsilon}) + \varepsilon^{2}\nabla p(\varrho_{\varepsilon})
    - 2\nu\textrm{div}(\varrho_{\varepsilon} \nabla w_{\varepsilon}) + 2\nu\textrm{div} (\varrho_{\varepsilon} \nabla v_{\varepsilon})
$$
\begin{equation} \label{ags_2.0}
    - 2\kappa^2 \left(
    \nabla (\varrho_{\varepsilon} \Delta \varrho_{\varepsilon}) + \frac{1}{2} \nabla ( |\nabla \varrho_{\varepsilon}|^2)
    - 4 \textrm{div}(\varrho_{\varepsilon} \nabla \sqrt{\varrho_{\varepsilon}} \otimes \nabla \sqrt{\varrho_{\varepsilon}})
    \right)=0,
\end{equation}
\begin{equation} \label{ags_3.0}
    \partial_t (\varrho_{\varepsilon} v_{\varepsilon}) + \textrm{div} (\varrho_{\varepsilon} v_{\varepsilon} \otimes u_{\varepsilon}) + 2\nu\textrm{div} (\varrho_{\varepsilon} \nabla^t u_{\varepsilon}) = 0.
\end{equation}
and the energy inequality reads as follows 
    \begin{equation}
   \label{ei-rs}
  \begin{split}
    \frac{d}{dt} \int_{\mathbb{T}^3} \Big( \frac{1}{2} \varrho_{\varepsilon}(t,x) \left( |w_{\varepsilon}(t,x)|^2 + |v_{\varepsilon}(t,x)|^2 \right) &+ \varepsilon^{2}H(\varrho_{\varepsilon}) + \kappa^2|\nabla \varrho_{\varepsilon}(t,x)|^2 \Big) dx\\
 +2\nu \int_{\mathbb{T}^3} \Big( |\mathcal{S}(u_{\varepsilon})(t,x)|^2 + |\mathcal{A}(w_{\varepsilon})(t,x)|^2   &+ \varepsilon^{2}\frac{p'(\varrho_{\varepsilon})}{\varrho_{\varepsilon}} |\nabla \varrho_{\varepsilon}(t,x)|^2 \Big)dx 
    \\
    +4\nu \kappa^2 \int_{\mathbb{T}^3} |\Delta \varrho_{\varepsilon}(t,x)|^2  dx
    \leq 0,
   \end{split}
\end{equation}
Now, the relative entropy functional  has the following structure,
$$
\mathcal{E}(T,\cdot) = \mathcal{E}(\varrho_{\varepsilon}, w_{\varepsilon}, v_{\varepsilon}|r, W, V)(T,\cdot)$$
$$=  \int_{\mathbb{T}^3} 
\left[
\frac{1}{2} \varrho_{\varepsilon} \left( |w_{\varepsilon}-W|^2 + |v_{\varepsilon}-V|^2 \right) 
+ \kappa^2 |\nabla \varrho_{\varepsilon} - \nabla r|^2  
\right]
dx 
+ \varepsilon^{2}\int_{\mathbb{T}^3} H(\varrho_{\varepsilon}|r) dx
$$
By using  the energy inequality \eqref{ei-rs}, we can write
$$
\mathcal{E}(T,\cdot) - \mathcal{E}(0,\cdot) $$
$$\leq 
\int_{\mathbb{T}^3} \left( \frac{1}{2} \varrho_{\varepsilon} |W|^2 - \varrho_{\varepsilon} w_{\varepsilon} \cdot W + \frac{1}{2}\varrho_{\varepsilon} |V|^2 - \varrho_{\varepsilon} v_{\varepsilon} \cdot V
+\kappa^2|\nabla r|^2 - 2\kappa^2\nabla \varrho_{\varepsilon} \nabla r\right) (T,\cdot)
dx
$$
$$
- \int_{\mathbb{T}^3} \left( \frac{1}{2} \varrho_{\varepsilon} |W|^2 - \varrho_{\varepsilon} w_{\varepsilon} \cdot W + \frac{1}{2}\varrho_{\varepsilon} |V|^2 - \varrho_{\varepsilon} v_{\varepsilon} \cdot
V
+\kappa^2|\nabla r|^2 - 2\kappa^2\nabla \varrho_{\varepsilon} \nabla r\right) (0,\cdot)
dx
$$
$$
-\varepsilon^{2} \int_{\mathbb{T}^3} \left( H(r) + H'(r)(\varrho_{\varepsilon}-r) \right) (T,\cdot)
dx
+ \varepsilon^{2}\int_{\mathbb{T}^3} \left( H(r) + H'(r)(\varrho_{\varepsilon}-r) \right) (0,\cdot)
dx
$$
$$
-2\nu\int_0^T\int_{\mathbb{T}^3} \left( |\mathcal{S}(u_{\varepsilon})|^2 + |\mathcal{A}(w_{\varepsilon})|^2   + \frac{p'(\varrho_{\varepsilon})}{\varrho_{\varepsilon}} |\nabla \varrho_{\varepsilon}|^2 \right)dx dt
$$
\begin{equation} \label{step-1rs}
-4\nu\kappa^2 \int_0^T\int_{\mathbb{T}^3} |\Delta \varrho_{\varepsilon}|^2  dxdt.
\end{equation}
Now we are ready to state our weak-strong convergence result in the high-Mach number limit regime.

\begin{theorem} \label{main-r3}
    Let $\mathbb{T}^3$ be a three-dimensional flat torus. Let us consider $(\varrho_{\varepsilon},u_{\varepsilon})$  a weak solution to the re-scaled Navier-Stokes-Korteweg system (\ref{cont_cap-rs}) - (\ref{mom_cap-rs}). Moreover assume that the initial data $(\varrho_{\varepsilon 0}, u_{\varepsilon 0})$ satisfy \eqref{id-th-1}, \eqref{id-th-2} and the following property
    \begin{equation}
    \label{prop-initial}
    \begin{split}
    \varrho_{\varepsilon 0} \to r_{0} \quad \text{strongly in $L^{\gamma}(\mathbb{T}^3)$}, &\qquad   
    \nabla\varrho_{\varepsilon 0} \to \nabla r_{0} \quad \text{strongly in $L^{2}(\mathbb{T}^3)$}\\ 
     \sqrt{\varrho_{\varepsilon 0}} u_{\varepsilon 0}&\to \sqrt{r_{0}}U_{0} \quad \text{strongly in $L^{2}(\mathbb{T}^3)$}.
    \end{split}
    \end{equation}
    Let $(r,U)$ be the strong solution of the  pressureless system (\ref{cont_cap-press}) - (\ref{mom_cap-press}) with initial data $(r_{0},U_{0})$. Then, as $\varepsilon \to 0$, we have that $(\varrho_{\varepsilon},u_{\varepsilon})$ converges to $(r,U)$, more precisely it holds that
    $$\sup_{t\in(0,T)}\|\sqrt{\varrho_{\varepsilon}}u_{\varepsilon}-\sqrt{r}U\|_{L^{2}(\mathbb{T}^3)}\rightarrow 0,$$
    $$\sup_{t\in(0,T)}\|\nabla\varrho_{\varepsilon}-\nabla r\|_{L^{2}(\mathbb{T}^3)}\rightarrow 0,$$
    $$\sup_{t\in(0,T)}\|\varrho_{\varepsilon}- r\|_{L^{2}+L^{\gamma}(\mathbb{T}^3)}\rightarrow 0.$$
    
\end{theorem}

\begin{proof}
In this case, the relative energy inequality compared to  \eqref{step-4} will present a slightly different form due to the fact that we take as test functions the strong solutions of the augmented pressureless system \eqref{cont-r-pless}-\eqref{mom-V-pless}. In recovering the relative entropy inequality we follow the same line of arguments as in Section \ref{ws-un} and we show the main differences. First, since we  don't  have the pressure term in (\ref{mom-W-pless}), relation (\ref{mom-W-test}) does not contain the term 
$$
\int_0^T \int_{\mathbb{T}^3}
\frac{\varrho_{\varepsilon}}{r}\nabla p(r)(W-w_{\varepsilon})
dxdt.
$$
Consequently, relation (\ref{rew}) presents the following form
$$
- \int_0^T \int_{\mathbb{T}^3} 
\partial_t (H'(r))(\varrho_{\varepsilon}-r)
dxdt
- \int_0^T \int_{\mathbb{T}^3} \varrho_{\varepsilon} u_{\varepsilon} \nabla(H'(r)) dxdt
+ \int_0^T \int_{\mathbb{T}^3} H'(r) \partial_t r dxdt
$$
$$
= \int_0^T \int_{\mathbb{T}^3}
    \varrho_{\varepsilon} \nabla H'(r)(W-w_{\varepsilon}) dxdt. 
$$
\begin{equation} \label{rew-hm}
    + \int_0^T \int_{\mathbb{T}^3} \varrho_{\varepsilon} \nabla (H'(r)) \cdot (v_{\varepsilon} - V)
    dxdt
    + \int_0^T \int_{\mathbb{T}^3} p'(r) \textrm{div}U (\varrho_{\varepsilon} - r)dxdt
\end{equation}
and the relative entropy inequality will contain one more term, namely
\begin{equation} \label{add-term}
    \int_0^T \int_{\mathbb{T}^3}
    \varrho_{\varepsilon} \nabla H'(r)(W-w_{\varepsilon}) dxdt. 
\end{equation}
More precisely, (\ref{step-5}) becomes,
$$
\mathcal{E}(T,\cdot) - \mathcal{E}(0,\cdot)\leq\int_0^T \int_{\mathbb{T}^3} \varrho_{\varepsilon} \nabla W (u_{\varepsilon}-U) \cdot (W-w_{\varepsilon})dxdt
$$
$$
+ \int_0^T \int_{\mathbb{T}^3} \varrho_{\varepsilon} \nabla V (u_{\varepsilon}-U)\cdot(V-v_{\varepsilon})dxdt
$$
$$
+ 2\nu \int_0^T \int_{\mathbb{T}^3}\frac{\varrho_{\varepsilon}}{r} \left(\textrm{div} (r \nabla W) \cdot (W-w_{\varepsilon})+ \textrm{div}(r \nabla^tV) \cdot (V-v_{\varepsilon})\right)
dxdt
$$
$$
- 2\nu \int_0^T \int_{\mathbb{T}^3}
\frac{\varrho_{\varepsilon}}{r} \left(\textrm{div} (r \nabla^t W) (V-v_{\varepsilon})+ \textrm{div} (r \nabla V) (W-w_{\varepsilon})\right)dxdt
$$
$$
+ 2\nu \int_0^T \int_{\mathbb{T}^3} \varrho_{\varepsilon} \left( \nabla^t v_{\varepsilon} : \nabla V + \nabla w_{\varepsilon}: \nabla W  \right)dxdt
- 2\nu \int_0^T \int_{\mathbb{T}^3} \varrho_{\varepsilon}\left(\nabla v_{\varepsilon}:\nabla W+  \nabla^t w_{\varepsilon}: \nabla V\right)dxdt
$$
$$
+ 2\kappa^2 \int_0^T \int_{\mathbb{T}^3} \left( \varrho_{\varepsilon} \partial_t \Delta r + \varrho_{\varepsilon} u_{\varepsilon} \cdot \nabla \Delta r - \partial_t r \Delta r + 
\varrho_{\varepsilon} \nabla \Delta r (W-w_{\varepsilon})\right)dxdt
$$
$$
- 2\kappa^2 \int_0^T \int_{\mathbb{T}^3} 
\left( \nabla\varrho_{\varepsilon} \cdot \nabla( \varrho_{\varepsilon} \textrm{div} W) - \frac{1}{2} |\nabla \varrho_{\varepsilon}|^2 \textrm{div} W+ \nabla \varrho_{\varepsilon} \otimes \nabla \varrho_{\varepsilon} : \nabla W\right)dxdt
$$
$$
+\varepsilon^2\int_0^T \int_{\mathbb{T}^3} \left( \left(  p'(r)(\varrho_{\varepsilon} - r) - p(\varrho_{\varepsilon}) + p(r) \right) \textrm{div}U\right)dxdt
$$
$$
-2\nu \int_0^T \int_{\mathbb{T}^3} \left( |\mathcal{S}(u_{\varepsilon})|^2 + |\mathcal{A}(w_{\varepsilon})|^2   +\varepsilon^2\frac{p'(\varrho_{\varepsilon})}{\varrho_{\varepsilon}} |\nabla \varrho_{\varepsilon}|^2 \right)dxdt -4\nu \kappa^2 \int_0^T \int_{\mathbb{T}^3} |\Delta \varrho_{\varepsilon}|^2  dxdt
$$
$$
- \varepsilon^2\int_0^T \int_{\mathbb{T}^3} 
\left(
\varrho_{\varepsilon} \nabla (H'(r)) \cdot (V - v_{\varepsilon})+ p(\varrho_{\varepsilon}) \textrm{div}V 
\right)
dxdt,
$$
\begin{equation}
 \label{step-5-eps}
+\varepsilon^2 \int_0^T \int_{\mathbb{T}^3}\varrho_{\varepsilon} \nabla H'(r)(W-w_{\varepsilon}) dxdt.
\end{equation}
where the multiplying factor $\varepsilon^2$ is due to the Mach number scaling. All the terms in the right hand side of (\ref{step-5-eps}) could be estimated as in Section \ref{ws-un} except for the last one which we handle as follows,
\begin{equation*}
    \varepsilon^2 \int_0^T \int_{\mathbb{T}^3}
    \varrho_{\varepsilon} \nabla H'(r)(W-w_{\varepsilon}) dxdt
\end{equation*}
\begin{equation*}
    \leq
    \frac{1}{2}\varepsilon^2 \int_0^T \int_{\mathbb{T}^3} \varrho_{\varepsilon} |W-w_{\varepsilon}|^2 dxdt
    +\frac{1}{2}\varepsilon^2 \int_0^T \int_{\mathbb{T}^3} \varrho_{\varepsilon} |\nabla H'(r)|^2 dxdt
\end{equation*}
\begin{equation*}
    \leq
    C(\varepsilon^2) \int_0^T \mathcal{E}(t,\cdot) dt
    +
    \varepsilon^2 C(T),
\end{equation*}
with $C(\varepsilon^2) \to 0$ as $\varepsilon \to 0$, and being $\varrho \in L^1(\Omega)$ and $\nabla H'(r)$ bounded. 
Putting all the steps together we obtain,
$$\mathcal{E}(T,\cdot) - \mathcal{E}(0,\cdot)\leq  \varepsilon^2 C(T)+C \int_0^T \mathcal{E}(t,\cdot) dt.$$
Finally,
$$\mathcal{E}(T,\cdot)\leq(\mathcal{E}(0,\cdot)+\varepsilon^2 C(T))(1+Te^{CT}).$$
and we conclude the proof of Theorem \ref{main-r3} by sending $\varepsilon \to 0$ and by taking into account  \eqref{prop-initial}.
\end{proof}

\bigskip
\noindent
\textbf{Acknowledgment}
M. Caggio has been supported by the Praemium Academiae of \v S. Ne\v casov\' a, and by the Czech Science Foundation under the grant GA\v CR 22-01591S. The Institute of Mathematics, CAS is supported by RVO:67985840. The work of D. Donateli was partially supported by the Ministry of University and Research (MUR), Italy under the grant PRIN 2020- Project N. 20204NT8W4, Nonlinear evolution PDEs, fluid dynamics and transport equations: theoretical foundations and applications


\begin{thebibliography}{}

\bibitem{AnSp} P. Antonelli, S.  Spirito, Global existence of weak solutions to the Navier-Stokes-Korteweg equations. Ann. Inst. H. Poincar\' e C Anal. Non Lin\' eaire 39 (2022), no. 1, 171–200.

\bibitem{BrDeLi}  D. Bresch, B. Desjardins, C-K. Lin, On some compressible fluid models: Korteweg, lubrication, and shallow water systems. Comm. Partial Differential Equations 28 (2003), no. 3-4, 843–868.

\bibitem{BrDeZa}  D. Bresch, B. Desjardins, E. Zatorska, Two-velocity hydrodynamics in fluid mechanics: Part II. Existence of global $\kappa$-entropy solutions to the compressible Navier-Stokes systems with degenerate viscosities. J. Math. Pures Appl. (9) 104 (2015), no. 4, 801–836.

\bibitem{BNV-1} D. Bresch, P. Noble, J.-P. Vila, Relative entropy
for compressible Navier-Stokes equations with density-dependent
viscosities and applications, C. R. Math. Acad. Sci. Paris,
354(1) (2016), 45-49.

\bibitem{BNV-2} D. Bresch, P. Noble, J.-P. Vila, Relative entropy for compressible Navier-Stokes equations with density dependent viscosities and various applications, LMLFN 2015--low velocity flows--application to low Mach and low Froude regimes, 40--57, ESAIM Proc. Surveys, 58, EDP Sci., Les Ulis (2017).

\bibitem{BGL} D. Bresch, M. Gisclon, I.  Lacroix-Violet, On Navier-Stokes-Korteweg and Euler-Korteweg systems: application to quantum fluids models. Arch. Ration. Mech. Anal. 233 (2019), no. 3, 975–1025.

\bibitem{CaDoNeSun}M. Caggio, D. Donatelli, S. Necasova, Y. Sun, Low Mach number limit on thin domains,  Nonlinearity 33 (2020), 840--863.

\bibitem{CaDo} M. Caggio, D. Donatelli, High Mach number limit for Korteweg fluids with density dependent viscosity. J. Differential Equations 277 (2021), 1–37.

\bibitem{CianLat} G. Cianfarani Carnevale, C. Lattanzio,  High friction limit for Euler-Korteweg and Navier-Stokes-Korteweg models via relative entropy approach. J. Differential Equations, 269 (2020),  10495--10526.

\bibitem{Daf} C. M. Dafermos, The second law of thermodynamics and stability,  Arch. Rational Mech. Anal. 70,  (1979),167--179.

\bibitem{DoFe} D. Donatelli, E. Feireisl, An anelastic approximation arising in astrophysics, Math. Ann. 369,  (2017), 1573--1597.

\bibitem{DoFeMa} D. Donatelli, E. Feireisl, P. Marcati, Well/ill posedness for the Euler-Korteweg-Poisson system and related problems, 
Comm. Partial Differential Equations, 40, (2015), 1314--1335. 

\bibitem{FeNoJi} E. Feireisl, B. J. Jin, A. Novotn\' y,  Relative entropies, suitable weak solutions, and weak-strong uniqueness for the compressible Navier-Stokes system. J. Math. Fluid Mech. 14 (2012), no. 4, 717–730.

\bibitem{HaLi94} H. Hattori, D. Li, Solutions for two dimensional system for materials of Korteweg type, SIAM J. Math. Anal., 25 (1994), 85--98.

\bibitem{HaLi96} H. Hattori, D. Li, Global solutions of a high dimensional system for Korteweg materials, J. Math. Anal. Appl., 198 (1996), 84--97.

\bibitem{BH-1} B. Haspot, From the highly compressible Navier--Stokes equations to fast diffusion and porous media equations, existence of global weak solution for the quasi-solutions, J. Math. Fluid Mech., 18, (2016), 243--291, .

\bibitem{BH-2} B. Haspot, Global existence of strong solution for viscous shallow water system with large initial data on the irrotational part. J. Differential Equations 262 (2017),  4931--4978.	
	
 \end{thebibliography}
\end{document}